\documentclass[11pt]{amsart}
\usepackage{amssymb, amscd}
\usepackage{eucal}
\usepackage{color}
\usepackage{enumerate}
\usepackage{soul}
\usepackage{blindtext}
\usepackage{hyperref}

\numberwithin{equation}{section}

\DeclareSymbolFont{SY}{U}{psy}{m}{n}
\DeclareMathSymbol{\emptyset}{\mathord}{SY}{'306}

\theoremstyle{plain}
\newtheorem*{mainT}{Theorem}

\newtheorem{thm}{Theorem}[section]
\newtheorem{cor}[thm]{Corollary}
\newtheorem{lem}[thm]{Lemma}
\newtheorem{prop}[thm]{Proposition}

\theoremstyle{definition}
\newtheorem{defn}[thm]{Definition}

\newtheorem{rem}[thm]{Remark}

\newcommand{\overbar}[1]{\mkern 1.5mu\overline{\mkern-1.5mu#1\mkern-1.5mu}\mkern 1.5mu}

\newenvironment{psmallmatrix}
 {\left(\begin{smallmatrix}}
 {\end{smallmatrix}\right)}
\newcounter{defcounter}
\title{Homogeneous $2$-shifts}

\author{Somnath Hazra}

\address{Department of Mathematics\\Indian Institutte of Science\\Bangalore 560012}
\email{somnathhazra@iisc.ac.in}

\thanks{The results in this paper are taken form the PhD thesis of the author submitted to the Indian Institute of Science in the year 2017. This work 
is supported by the Council of Scientific and Industrial Research (CSIR), India.}

\begin{document}

\begin{abstract}

The classification of homogeneous scalar weighted shifts is known. Recently, Kor\'{a}nyi obtained a large class of inequivalent irreducible homogeneous bi-lateral $2$-by-$2$ block shifts. In this paper, we construct two distinct classes of examples not in the list of Kor\'{a}nyi. It is then shown that these new examples of irreducible homogeneous bi-lateral $2$-by-$2$ block shifts, together with the ones found earlier by Kor\'{a}nyi, account for every unitarily inequivalent irreducible homogeneous bi-lateral $2$-by-$2$ block shift.
\end{abstract}

\maketitle
\section{Introduction}
Let M\"{o}b denote the M\"{o}bius group of all biholomorphic automorphisms $\phi$ of the unit disc $\mathbb{D}:=\{z\in \mathbb C: |z| < 1\}$. These are of the form $\phi(z) = e^{i\theta} \tfrac{z-a}{1-\bar{a}z},$ $\theta \in \mathbb R,\,a\in \mathbb D.$   
\begin{defn}
A bounded linear operator $T$ on a complex separable Hilbert space $H$ is said to be homogeneous if the spectrum of $T$ is contained in $\overline{\mathbb{D}},$ the closed unit disc and $\phi(T)$ is unitarily equivalent to $T$ for every $\phi$ in M\"{o}b. 
\end{defn}
These assumptions on an operator $T$ and the Hilbert space $H,$ namely that the operator is linear and bounded, the Hilbert space is complex and separable will be in force throughout this paper. 

If $T$ is an irreducible homogeneous operator, then there exists a unique (up to equivalence) projective representation $\pi$ of M\"{o}b such that $\phi(T) = \pi(\phi)^* T \pi(\phi),\,\,\phi \in \mbox{M\"{o}b}.$ It was shown in \cite{Homogshift} that any irreducible homogeneous operator must be a block shift. All homogeneous scalar shifts were listed in \cite{Homogshift}. Kor\'anyi discovered the first examples of irreducible homogeneous bi-lateral $2$-by-$2$ block shifts. All the unilateral $n$-by-$n$ block shifts in the Cowen-Douglas class have been listed in \cite{ClassCD}. 

The classification result in \cite{Homogshift} was obtained by identifying the associated representation $\pi$ of a fixed but arbitrary homogeneous scalar shift. In this paper, we adopt this technique to the case of homogeneous $2$-by-$2$ block shifts. The possibilities for the associated projective representation, in this case, are given in Section $3$. Picking any one of these representation, say $\pi$, we determine the set
$$\{T : \phi(T) = \pi(\phi)^* T \pi(\phi),\,\,\phi \in \mbox{M\"ob}\}.$$
This is achieved by dividing the list of projective representation associated with $T$ according to the number of irreducible components in it. There are three such possibilities which are given in Theorem \ref{description of the rep for a 2 shift}. We show that the operator $T$ for which the associated representation is either a direct sum of three or four irreducible representations as described in the second and the third case of the Theorem \ref{description of the rep for a 2 shift}, is reducible.

In the remaining case, where the associated representation is a direct sum of two irreducible Continuous series representations, some of the operators $T$ that occur were already discovered by Kor\'anyi. We complete the list by identifying the remaining irreducible homogeneous bi-lateral $2$-by-$2$ block shifts. The main Theorem of this paper is stated at the very end of the paper.

\section{Preliminaries}
The definition of homogeneous operator while ensuring the existence of  a unitary operator $U_{\phi}$ intertwining $\phi(T)$ with $T$ does not impose any additional condition on the map $\phi \mapsto U_{\phi}$. To investigate some of these properties, we recall some basic notions from representation theory of locally compact second countable (lcsc) groups, in particular, the M\"{o}bius group. 
Most of what follows is from \cite{Homogshift, Homogsurvey}.
\begin{defn}
Let $G$ be a locally compact second countable group, $H$ be a Hilbert space and $\mathcal{U}(H)$ be the group of unitary operators on $H$. A Borel function $\pi : G \to \mathcal{U}(H)$ is said to be a projective unitary representation of $G$ on the Hilbert space $H$, if 
$$\pi(1) = I,\,\,\,\pi(gh) = m(g, h) \pi(g) \pi(h);\,\,g, h \in G,$$
where $m : G \times G \to \mathbb{T}$ is a Borel function. (In this paper, a representation or a projective representation will always mean a projective unitary representation.)
\end{defn}
The function $m$ associated with a projective representation $\pi$ is called the \emph{multiplier} of $\pi$ and satisfies the equations
$$m(g, 1) = m(1, g) = 1,\,\,m(g_1, g_2) m(g_1 g_2, g_3) = m(g_1, g_2 g_3) m(g_2, g_3)$$
for all $g, g_1, g_2$ and $g_3$ in $G$. Two multipliers $m$ and $\tilde{m}$ are said to be equivalent if there is a Borel function $f : G \to \mathbb{T}$ such that $m(g, h) = \frac{f(gh)}{f(g)f(h)} \tilde{m}(g, h),\,\,g, h \in G.$

Let $\pi_1$ and $\pi_2$ be two projective representations of $G$ on Hilbert spaces $H_1$ and $H_2,$ respectively. The representations $\pi_1$ and $\pi_2$ are called equivalent if there exists a unitary operator $U : H_1 \to H_2$ and a Borel function $f : G \to \mathbb{T}$ such that
$$\pi_1(g) = f(g) U^* \pi_2(g) U$$
holds for all $g$ in $G$.

\begin{defn}
Let $T$ be a homogeneous operator on a Hilbert space $H$. If there is a projective representation $\pi$ of M\"{o}b on $H$ with the property 
$$\phi(T) = \pi(\phi)^* T \pi(\phi),\, \phi \in \mbox{M\"{o}b},$$
then $\pi$ is said to be the representation associated with the operator $T.$
\end{defn}
A homogeneous operator need not possess an associated representation. 
However, \cite[Theorem 2.2]{Homogshift} says that for every irreducible homogeneous operator, there exists a unique (upto equivalence) projective representation associated with it.

We fix some notation and terminology that will be used throughout this paper. For any projective representation $\pi$ of M\"{o}b, let $\pi^{\#}$ be the representation of M\"{o}b defined by $\pi^{\#}(\phi) = \pi(\phi^*)$ where $\phi^*(z) = \overbar{\phi(\bar{z})},\,z \in \mathbb{D},$ for every $\phi$ in M\"{o}b. 

\begin{prop}\cite[Proposition 2.1]{Homogshift}\label{proposition 1}
Suppose $T$ is a homogeneous operator and $\pi$ is an associated representation of $T$. Then the adjoint, $T^*,$ is also homogeneous and $\pi^{\#}$ is an associated representation of $T^*.$ If $T$ is invertible, then $T^{-1}$ is also homogeneous and $\pi^{\#}$ is an associated representation of $T^{-1}.$ In particular $T$ and ${T^*}^{-1}$ have the same associated representation.
\end{prop}

A complete list of irreducible projective representations of M\"{o}b is given in  \cite[List 3.1]{Homogshift}. 
The Complementary series representations and the Principal series representations are together called Continuous series representations.

A projective representation $\pi$ of M\"{o}b on a Hilbert space $H$, containing a dense subspace $\mathcal{M}$ consisting of functions on some set $X$, is called a multiplier representation if
$$\left( \pi(\phi^{-1}) f\right)(x) = c(\phi, x)(f \circ \phi)(x),\,\,\phi \in \mbox{M\"{o}b},\,\,f \in \mathcal{M},\,\,x \in X$$
where $c$ is a non-vanishing measurable function on $\mbox{M\"{o}b} \times X.$

\begin{thm}\cite[Theorem 2.3]{Homogshift} \label{theorem 2}
Suppose there is a multiplier representation $\pi$ of M\"{o}b on a Hilbert space $H$, containing a dense subspace $\mathcal{M}$ consisting of functions on some set $X$. Suppose the operator $T$ given on $\mathcal{M}$ by 
$$(Tf)(x) = xf(x),\,\,f \in \mathcal{M},\,\,x \in X$$
leaves $\mathcal{M}$ invariant and has a bounded extension to $H$. Then the extension of\, $T$ is homogeneous and $\pi$ is associated with $T$.
\end{thm}
From the list of the irreducible projective representations of M\"{o}b (\cite[List 3.1]{Homogshift}), we see that every irreducible projective representation of M\"{o}b is a multiplier representation. Therefore Theorem \ref{theorem 2} says that if the multiplication by the coordinate function on the representation space of an irreducible projective representation of M\"{o}b is bounded, then it must be homogeneous.  Indeed this is true and  a complete list of homogeneous operators is given in \cite[List 4.1]{Homogshift}.

A bounded operator $T$ on a Hilbert space $H$ is said to be a shift if $H$ admits a direct sum decomposition of the form $\displaystyle \oplus_{i \in I} H_{i},$ where each $H_i$ is a closed subspace of $H$ and $T$ maps $H_{i}$ into $H_{i+1}$,  $i \in I.$  The operator $T$ is a bi-lateral, forward or backward shift according as $I$ equals $\mathbb Z,$ $\{n\in\mathbb Z: n\geq n_0\}$ or $\{n\in \mathbb Z: n\leq n_0\}.$
If there is a decomposition of the Hilbert space on which the  operator $T$ acts as a shift and if $T$ is irreducible, then this decomposition must be unique (see \cite[Lemma 2.2]{Homogshift}).
\begin{defn}
An irreducible operator $T$ is said to be an $n$-shift if  dim $H_{i} = n$, for all $i \in I$ except for finitely many of them. In the paper of Kor\'{a}nyi \cite{HomogKoranyi}, the $2$ shifts were called $2$-by-$2$ block shifts. 
\end{defn}

All irreducible homogeneous forward (and consequently backward) $n$-shifts are described  in \cite{Multfree}. First example of an irreducible homogeneous bilateral $2$-shift was given by Kor\'anyi in \cite{HomogKoranyi}. In \cite{HomogKoranyi}, a three parameter family of irreducible homogeneous bilateral $2$-shifts was constructed by Kor\'{a}nyi using \cite[Lemma 2.1]{HomogKoranyi},  
which also follows by combining  \cite[Theorem 5.3]{Homogsurvey} and \cite[Proposition 2.4]{CONSTCharGMBB}.

Recall from \cite[Lemma 2.1]{HomogKoranyi} that if $\pi(\phi)^* T_i \pi(\phi) = \phi(T_i),\,i=1,2,$ for some representation of M\"{o}b, then the operator $\begin{psmallmatrix} T_1 & \alpha (T_1 - T_2) \\ 0 & T_2 \end{psmallmatrix},\,\alpha > 0,$ is homogeneous. From the list of homogeneous weighted shifts (\cite[List 4.1]{Homogshift}), we see that for $0< a < b < 1,$ the bi-lateral shifts $T(a, b)$ and $T(b, a)$ with weights  $\sqrt{\tfrac{n+a}{n+b}}$ and $\sqrt{\tfrac{n+b}{n+a}},$ respectively, are homogeneous and the associated representation is the Complementary series  $\pi = C_{\lambda, \sigma},$ 
where $\lambda= a+b-1$ and $\sigma = (b-a)/2.$ Consequently, the operator $\begin{psmallmatrix} T(a, b) & \alpha (T(a, b) - T(b, a)) \\ 0 & T(b, a) \end{psmallmatrix},\,\,\alpha \in \mathbb{C},$ is homogeneous. In \cite{HomogKoranyi}, Kor\'anyi shows that the family
$$\mathcal C:=\left \lbrace T(a,b, \alpha) = 
                  \left[\begin{array}{ccc}
                    T(a,b) & \alpha (T(a,b) - T(b,a))\\
                      0 & T(b,a)
                  \end{array}\right] : 0 < a < b < 1, \alpha > 0 \right \rbrace$$
contains all irreducible homogeneous operators, modulo unitary equivalence, whose associated representation is $C_{\lambda,\sigma} \oplus C_{\lambda,\sigma}.$ 
In this paper, we describe all irreducible homogeneous $2$-shifts up to unitary equivalence completing the list of irreducible homogeneous $2$-shifts of Kor\'anyi.

\section{Representation associated with an irreducible homogeneous $2$-shift}

In this section, we describe the associated representation of an irreducible homogeneous $2$-shit. 
Let $\mathbb K$ be the maximal compact subgroup consisting of those elements of M\"{o}b which fix the point $0.$ Recall that a subspace 
$$V_{n}(\pi):=\{h:  \pi(k) h = k^{-n} h,\,\, k\in \mathbb K\}$$ 
of the representation space $H$ is said to be $\mathbb{K}$-isotypic.
Setting 
$I(\pi) = \{n \in \mathbb{Z} : \dim ~V_{n}(\pi) \neq 0 \},$ we note that the operator $T$ must be a shift from $V_{n}(\pi)$ to $V_{n+1}(\pi),$ $n \in I(\pi)$ by virtue of \cite[Theorem 5.1]{Homogshift}. The set $I(\pi)$ is said to be connected if for any three elements $a, b, c$ in $\mathbb{Z}$ with $a < b < c$ and $a, c \in I(\pi)$, then $b \in I(\pi)$ (see \cite[Definition 3.5]{Homogshift}).

\begin{thm}\label{decomposition of rep for a rank 2 shift}
Suppose $T$ is an irreducible $2$-shift  homogeneous operator. Then the associated representation $\pi$ is the direct sum of two or three or at most four irreducible representations. 
\end{thm}

\begin{proof} Let $T$ be an  irreducible homogeneous $2$ - shift  and $\pi$ be the associated representation.  

Since the $\mathbb{K}$-isotypic subspace of an irreducible projective representation is one dimensional (cf. \cite[Theorem 5.1]{Homogshift}), it follows that  $\pi$ cannot be irreducible.  

Thus we may assume without loss of generality that $\pi$ is a direct sum of two non-trivial representations, say, $\pi_{00} \oplus \pi_{22}$. If both of them are irreducible, then we are done. 

If not, one of them, say, $\pi_{00}$ must be reducible. Then $\pi_{00}$ is the  direct sum of two non-trivial representations, namely,  $\pi_{00} = \pi_{01} \oplus \pi_{21}$.   Hence $\pi = \pi_{01} \oplus \pi_{21} \oplus \pi_{22}$. 
If all of them are irreducible, then we are done. 

If not, one of them, say $\pi_{01}$, is reducible. Then $\pi_{01}$ is the
direct sum of two non-trivial representations, namely, $\pi_{01} = \pi_{11} \oplus \pi_{12}$. Then 
$$\pi = \pi_{11} \oplus \pi_{12} \oplus \pi_{21} \oplus \pi_{22}.$$ 

Now, we claim that  each summand in $\pi$ must be irreducible. If not, then one of them, say, $\pi_{11}$ is reducible. Then $\pi_{11} = \sigma \oplus \rho,$ where $\sigma$ and $\rho$ are non-trivial representations. Therefore, we have the decomposition

$$\pi = \sigma \oplus \rho \oplus \pi_{12} \oplus \pi_{21} \oplus \pi_{22}.$$ 

Now \cite[Lemma 3.2]{Homogshift} says that connected component of each $I(\sigma),$  $I(\rho),$  $I(\pi_{12}),$  $I(\pi_{21})$ and $I(\pi_{22})$ is unbounded. Therefore, each of $I(\sigma),$ $ I(\rho),$  $I(\pi_{12}),$ $I(\pi_{21})$ and $I(\pi_{22})$ contains a tail of $\mathbb{Z}$. This implies that one tail of $\mathbb{Z}$ must occur three times. Therefore,  $\dim V_{n}(\pi) \geq 3$ for all those $n$ in that tail of $\mathbb{Z}$ which occurs three times in $I(\pi).$ This contradicts the assumption that the operator $T$ is a $2$ - shift. Therefore each of $\pi_{11}, \pi_{12}, \pi_{21}$ and $\pi_{22}$ must be irreducible.
\end{proof}

The following theorem lists the possibilities of the associated representation for an irreducible homogeneous $2$-shift $T.$ 

\begin{thm}\label{description of the rep for a 2 shift}
If $T$ is an irreducible homogeneous $2$-shift and $\pi$ is the associated representation, then $\pi$ must be of the form
\begin{enumerate}
\item[]
$\boldsymbol{\pi = \displaystyle{\oplus_{i=1}^2} \pi_i :}$ In this case, the only possibilities for $\pi_1$ and $\pi_2$ are that they must be simultaneously one of the holomorphic Discrete series, anti-holomorphic Discrete series or Continuous series representations. 
\item[]
$\boldsymbol{\pi = \displaystyle\oplus_{i=1}^3 \pi_i :}$ In this case, one of the summands must be a Continuous series representation. Among the other two, one of them must be a holomorphic Discrete series and the other one an anti-holomorphic Discrete series representation.
\item[]
$\boldsymbol{\pi = \displaystyle \oplus_{i=1}^4 \pi_i :}$ In this case, two of the summands must be holomorphic Discrete series representations while the other two summands must be anti-holomorphic Discrete series representations simultaneously.
\end{enumerate}
\end{thm}

\begin{proof}
Suppose $\pi$ is a direct sum of two irreducible representations, say, $\pi = \pi_{1} \oplus \pi_{2}$. If one of them is a holomorphic Discrete series representation then the other one also has to be a holomorphic Discrete series representation. Suppose not, then there is at least one tail $\mathcal I$ of $\mathbb{Z}$ such that the dimension of $V_{n}(\pi),$ $n\in \mathcal I,$ is one.  Similarly, if one of them is an anti-holomorphic Discrete series representation then the other one has to be an anti-holomorphic Discrete series representation. 
It follows that if one of these representations is from the Continuous series, then the other one cannot be either the holomorphic or the anti-holomorphic Discrete series representation. This completes the proof of the first case.   

If $\pi$ is a direct sum of three irreducible representations, then one of them must be from the Continuous series representations.  If not, all the three summands are from the Discrete series representations. In consequence, the existence of a tail $\mathcal{I}$ in $\mathbb Z$ such that 
dimension of $V_{n}(\pi),$ $n\in \mathcal I,$ is either one or three follows. This contradiction proves our claim. If one of the summands is a Continuous  
series representation, then the other two cannot be simultaneously holomorphic or anti-holomorphic Discrete series representations. If not, we find  a tail $\mathcal I$ in $\mathbb Z$ for which the dimension of $V_{n}(\pi),$ $n\in \mathcal I,$ is $3,$ which is a contradiction.

Now suppose $\pi$ is the direct sum of four irreducible representations, say $\pi_{1} \oplus \pi_{2} \oplus \pi_{3} \oplus \pi_{4}$. If one of them is a Continuous series representation, then there exists a tail of $\mathbb{Z}$ for which the dimension of $V_{n}(\pi)$ is greater than or equal to $3$. So, none of the representations $\pi_{i},$ $1\leq i \leq 4,$ are from the Continuous series representations.  Thus  each one of the representations $\pi_{1}, \pi_{2}, \pi_{3}, \pi_{4}$ must be from the Discrete series.   Now if three of them are either from the holomorphic or anti-holomorphic Discrete series representations, then the dimension of $V_{n}(\pi)$ must be greater than or equal to 3 for $n$ in some tail of $\mathbb Z$. Therefore if $\pi$ is a direct sum of four irreducible representation, then two of them have to be holomorphic Discrete series representations and the other two have to be anti-holomorphic Discrete series representations.
\end{proof}

Each of the three cases enumerated in  Theorem \ref{description of the rep for a 2 shift} is analysed in the following Sections. 

\section{The associated representation is the direct sum of two representations from the Continuous series }In this section, we find all the irreducible homogeneous operators for which the associated representation is a direct sum of two Continuous series   representations. This naturally splits into several cases. In the paper \cite{HomogKoranyi}, the case when the associated representation $\pi$ is the direct sum $C_{\lambda,\sigma} \oplus C_{\lambda,\sigma},$ is discussed. Here we begin with the case when $\pi= \pi_1 \oplus \pi_2,$ $\pi_1, \pi_2$  are form the Principal series. 

\subsection{$\boldsymbol{\pi = P_{\lambda,s} \oplus P_{\lambda,s}}$}
In this subsection, we find all the irreducible homogeneous operators for which the associated representation $\pi$ is of the form $P_{\lambda,s} \oplus P_{\lambda,s}.$  It is convenient to separate two cases, namely, the case of $s=0$ and that of $s\not=0.$

%

\subsubsection{\bf The case ``$\mathbf{s\not = 0}$'':}  In what follows, we assume $s\not = 0.$
Let $B(s)$  be the bounded linear transformation on $L^{2}(\mathbb{T})$ obtained by requiring that 
$$B(s)\, z^n =  \frac{n + \frac{1 + \lambda}{2} + s}{n + \frac{1 + \lambda}{2} - s} z^{n+1},\,\, n\in \mathbb Z.$$ 
Thus it is the weighted bilateral shift with weight sequence $\left\lbrace w_{n} = \frac{n + \frac{1 + \lambda}{2} + s}{n + \frac{1 + \lambda}{2} - s} \right\rbrace.$  Let $B$ be the multiplication by the coordinate function  on $L^{2}(\mathbb{T})$. The operator $B$ is the unweighted bi-lateral shift. 
Both the operators $B(s)$ and $B$ are known to be homogeneous \cite[Theorem 5.2]{Homogshift}.  Each of the Principal series representations may be taken to be the associated representation for both of these operators. 
\begin{prop}\label{Finding S}
Let $S$ be an operator on $L^{2}(\mathbb{T}).$ Suppose that for all $\phi$ in M\"{o}b, we have  
\begin{equation} \label{eqn:2.1}
S P_{\lambda, s}(\phi) - e^{i \theta} P_{\lambda, s}(\phi) S = \overline{a} B(s) P_{\lambda, s}(\phi) S + \overline{a} S P_{\lambda, s}(\phi) B.
\end{equation}
Then $S = \alpha \left( B(s) - B \right)$ for some $\alpha\in \mathbb C.$ 
\end{prop}

\begin{proof}
Using homogeneity of $B(s)$ and $B$, it is easy to see that $\alpha (B(s) - B)$ satisfies \eqref{eqn:2.1} for all $\alpha \in \mathbb{C}$. We show that these operators are the only solutions of the equation \eqref{eqn:2.1}. 


For the proof, let  $S$  be any operator for which \eqref{eqn:2.1} holds. From the equation \eqref{basic eq: 3}, in the Appendix, it follows that $S$ is a weighted shift operator with respect to the orthonormal basis $\{z^{n}\}$ in $L^2(\mathbb T)$. Let $\{\alpha_{n}\}$ be the weight sequence of $S$. Now we find the value of $\alpha_{n}$. Now putting $m = n$ in the equation \eqref{basic eq: 4} and then comparing the coefficient of $r$, we get

\begin{equation*}
\alpha_{n-1} ( \lambda + 2n - 1 - 2s) = \alpha_{n}(\lambda + 2n + 1 - 2s) .
\end{equation*}
An easy induction argument shows that $\alpha_{n} = \alpha \left( \frac{n + \frac{1+ \lambda}{2} + s}{n + \frac{1+ \lambda}{2} - s} - 1 \right)$ for some $\alpha \in \mathbb{C}.$ This shows that $S = \alpha(B(s) - B)$ for some $\alpha \in \mathbb{C}$.

\end{proof}

\begin{cor}\label{homog S}
The operator $\left[ \begin{array}{ccc}
                    B(s) & \alpha (B(s) - B)\\
                      0 & B
                  \end{array}\right]$ is homogeneous with associate representation 
                  $P_{\lambda,s} \oplus P_{\lambda,s}.$
\end{cor}

\begin{proof}
The proof follows from \cite[Lemma 2.1]{HomogKoranyi}.
\end{proof}
It is evident that  $\left[ \begin{smallmatrix}
                    B(s) & \alpha (B(s) - B)\\
                      0 & B
                  \end{smallmatrix}\right]$ and $\left[ \begin{smallmatrix}
                    B(s) & \beta (B(s) - B)\\
                      0 & B
                  \end{smallmatrix}\right]$ are unitarily equivalent when $| \alpha | = | \beta |$.  A particular case of what is proved in \cite[Lemma 1.1]{HomogKoranyi} is that  $\left[ \begin{smallmatrix}
                    B(s) & \alpha (B(s) - B)\\
                      0 & B
                  \end{smallmatrix}\right]$ and $\left[ \begin{smallmatrix}
                    B & \alpha (B - B(s))\\
                      0 & B(s)
                  \end{smallmatrix}\right]$ are unitarily equivalent. We show that these are irreducible, which is very similar to the proof of \cite[Theorem 1.2]{HomogKoranyi}.

\begin{thm}\label{Irrducibility of operators corresponding to principal series rep}
For a fixed but arbitrary $\alpha > 0$, the operator  $T :=\left[ \begin{array}{ccc}
                    B(s) & \alpha (B(s) - B)\\
                      0 & B
                  \end{array}\right]$ is irreducible.
\end{thm}

\begin{proof}
Let $H(n)$ be the subspace of $L^{2}(\mathbb{T}) \oplus L^{2}(\mathbb{T})$  spanned  by the orthonormal set  
$$\mathcal{B}_{n} = \left\lbrace \left( \begin{array}{ccc}
                    z^{n}\\
                      0
                  \end{array}\right), \left( \begin{array}{ccc}
                    0\\
                   z^{n}
                  \end{array}\right) \right\rbrace.$$
Clearly, $T$ sends $H(n)$ to $H(n+1)$. Let $T_{n} := T_{|H(n)}.$ The matrix representations of $T_{n},\,T_n^*$ with respect to $\mathcal{B}_{n}$ and $\mathcal{B}_{n+1}$ are of the form 
$$\left[ \begin{array}{ccc}
                    w_{n} & \alpha (w_{n} - 1)\\
                      0 & 1
                  \end{array}\right]\,\, \mbox{\rm and}\,\, \left[ \begin{array}{ccc}
                    \bar{w}_{n} & 0\\
                     \alpha (\bar{w}_{n} - 1) & 1
                  \end{array}\right],$$  respectively.
The operators $A_{n} = T_{n}^{*}T_{n}$ and $B_{n} = T_{n-1}T_{n-1}^{*}$ map $H(n)$ to $H(n)$, their matrix representation with respect to the orthonormal basis $\mathcal{B}_{n}$ is easy to compute, namely, 
$$A_{n} = \left[ \begin{array}{ccc}
1 & \alpha (1 - \bar{w}_{n})\\
\alpha(1 - w_{n}) & 1 + \alpha^{2} | w_{n} - 1 |^{2}
\end{array}\right]\,\, \mbox{\rm and}\,\, B_{n} = \left[ \begin{array}{ccc}
       1 + \alpha^{2} | w_{n-1} - 1 |^{2} & \alpha (w_{n-1} - 1)\\
        \alpha (\bar{w}_{n-1} - 1) & 1
                                  \end{array}\right].$$ 
Since determinant of $A_{n}$ is $1$ and $A_{n}\not = I,$ it follows that the  eigenvalues of $A_{n}$ are of the form $\lambda_{n}^2, \frac{1}{\lambda^2_{n}}$ for some real number $\lambda_{n} > 1$. Consequently,  the trace of $A_{n}$ is $\lambda_{n}^{2} + \frac{1}{\lambda_{n}^{2}}.$ Thus $\lambda_{n}^{2} + \frac{1}{\lambda_{n}^{2}} = 2 + \alpha^2 | w_{n} - 1 |^{2}$ and therefore $\left( \lambda_{n} - \frac{1}{\lambda_{n}}\right)^{2} = \alpha^2 | w_{n} - 1 |^{2}$. 

Now suppose there exists $n , m$ such that $|w_{n} - 1|^{2} = |w_{m} - 1|^{2}$. Then putting the value of $w_{n}$ and $w_{m}$, we get $| n + \frac{1+\lambda}{2} - s|^{2} = | m + \frac{1+\lambda}{2} - s|^{2}$. Since $s = ia$, equivalently,  $\left(n + \frac{1+\lambda}{2}\right)^{2} + a^{2}  = \left(m + \frac{1+\lambda}{2}\right)^{2} + a^{2}$ and it follows that $n = m$ or $n+m+1+\lambda = 0$. Consequently, if $\lambda$ is not an integer, then $\lambda_{n} \neq \lambda_{m}$ for $n \neq m$. Since $-1 < \lambda \leq 1$, the possible integer values of $\lambda$ are either $0$ or $1$. If $\lambda = 0$ then $\lambda_{n} = \lambda_{-n-1}$ and if $\lambda = 1$, then $\lambda_{n} = \lambda_{-n-2}$. Note that $\lambda_{n} \neq \lambda_{m}$ if $n \neq m$ and $n, m \geq 0$. 
Let $\lambda_{n}^{(1)} = \lambda_{n}^{2},$ and $\lambda_{n}^{(2)} = \frac{1}{\lambda_{n}^{2}}.$ Pick an orthonormal basis $\{v_{n}^{(1)}, v_{n}^{(2)}\}$  of $H(n)$ which makes $A_{n}$ diagonal. Then $A_{n} v_{n}^{(i)} = \lambda_{n}^{(i)} v_{n}^{(i)}$. Let $u_{n}^{(i)} = T_{n-1}v_{n-1}^{(i)}$. Then $B_{n}u_{n}^{(i)} = T_{n-1}T_{n-1}^{*}T_{n-1}v_{n-1}^{(i)} = \lambda_{n-1}^{(i)} T_{n-1}v_{n-1}^{(i)} = \lambda_{n-1}^{(i)} u_{n}^{(i)}$. Also it is easily checked that $u_{n}^{(1)}$ and $u_{n}^{(2)}$ are orthogonal. So, $\{u_{n}^{(1)}, u_{n}^{(2)}\}$ is an orthogonal basis of $H(n)$ which makes $B_{n}$ diagonal.
                                  
Suppose $u_{n}^{(1)} = cv_{n}^{(1)}$ for $c\in \mathbb C$, then we  show that $u_{n}^{(2)} = dv_{n}^{(2)}$ for some $d\in \mathbb C$.  Find $d_1,d_2\in \mathbb C$ such that $u_{n}^{(2)} = d_{1} v_{n}^{(1)} + d_{2}v_{n}^{(2)}.$  Taking inner product of $v_{n}^{(1)}$ with $u_{n}^{(2)},$  we see that $d_{1} = 0$ using the equality $v_{n}^{(1)} = \frac{1}{c}u_{n}^{(1)}$ and the orthogonality of the two vectors $u_{n}^{(1)},$  $u_{n}^{(2)}.$ Thus we conclude that $u_{n}^{(2)}$ is a scalar multiple of $v_{n}^{(2)}$.

Similarly, we can show that if $u_{n}^{(1)}$ is a scalar multiple of $v_{n}^{(2)}$, then $u_{n}^{(2)}$ is a scalar multiple of $v_{n}^{(1)}$. This shows that if one of $\{v_{n}^{(1)}, v_{n}^{(2)}\}$ is a scalar multiple of one of $\{u_{n}^{(1)}, u_{n}^{(2)}\}$, then the same is true of the other one. 
If this  statement  is true for all $n,$ then we must have $A_{n}B_{n} - B_{n}A_{n} = 0$ for all $n$. But an easy computation  shows that $A_{n}B_{n} \neq B_{n}A_{n}$ for any $n \geq 1$.

Now let $\mathcal{K}$ be a reducing subspace of $T$. Then $\mathcal{K}$ is an invariant subspace of both $TT^{*}$ and $T^{*}T$ and therefore, for $f \in \mathcal{K}$, the projections of $f$ onto any eigenspaces of $TT^{*}$ and $T^{*}T$ are also in $\mathcal{K}$.

\begin{description}
\item[$\boldsymbol{\lambda \neq 0, 1}$]
Let  $\mathcal{A}_{n,i}$ be the space spanned  by the vector $v_{n}^{(i)}$. It is the eigenspace of $T^{*}T$ with eigenvalue $\lambda_{n}^{(i)}$. Then $L^{2}(\mathbb{T}) \oplus L^{2}(\mathbb{T}) = \displaystyle \oplus_{n \in \mathbb{Z}, i=1,2} \mathcal{A}_{n,i}$. Let $f \in \mathcal{K}$. Then $f = \displaystyle \sum _{n \in \mathbb{Z}, i = 1,2} \alpha_{n,i} v_{n}^{(i)}$. Since $f$ is non-zero, we can find $n, i$ such that $\alpha_{n,i} \neq 0$. Therefore, the vector $v_n^{(i)}$ is in $\mathcal K.$  This implies that $\mathcal{K} \cap H(n) \neq \emptyset$, for some $n \in \mathbb{Z}$.

\item[$\boldsymbol{\lambda = 0}$]
Let $\mathcal{A}_{n,i}$ be the space spanned  by the two vectors $v_{n}^{(i)}, v_{-n-1}^{(i)}.$ It is the eigenspace of $T^{*}T$ with  eigenvalue $\lambda_{n}^{(i)}$. Then $L^{2}(\mathbb{T}) \oplus L^{2}(\mathbb{T}) = \displaystyle \oplus_{n \geq 0, i=1,2} \mathcal{A}_{n,i}$. Now suppose $f \in \mathcal{K}$. Then 
$$f = \displaystyle \sum_{n \geq 0, i=1,2} \alpha_{n,i} h_{n,i},$$ 
where $h_{n,i}$ is in $\mathcal{A}_{n,i}$. Since $f \neq 0$, we can find $\alpha_{n,i} \neq 0$ for some $n, i$. Also there exist scalars $\gamma, \delta$ such that $h_{n,i} = \gamma v_{-n-1}^{(i)} + \delta v_{n}^{(i)}$. Since $v_{-n-1}^{(i)} \in H(-n-1)$ and $v_{n}^{(i)} \in H(n)$, applying $T^{n+2}$ we see that $T^{n+2} h_{n,i} = \tilde{\gamma} h_{1} + \tilde{\delta} h_{2n+2}$ for some $h_{1} \in H(1)$ and $h_{2n+2} \in H(2n+2)$. 
Therefore, there are  scalars $\gamma_{1}, \gamma_{2}, \delta_{1}, \delta_{2}$, such that $h_{1} = \gamma_{1}v_{1}^{(1)} + \gamma_{2}v_{1}^{(2)}$ and $h_{2n+2} = \delta_{1}v_{2n+2}^{(1)} + \delta_{2}v_{2n+2}^{(2)}$. So, 
$$T^{n+2}h_{n,i} = \tilde{\gamma}\gamma_{1} v_{1}^{(1)} + \tilde{\gamma}\gamma_{2} v_{1}^{(2)} + \tilde{\delta}\delta_{1}v_{2n+2}^{(1)} + \tilde{\delta}\delta_{2}v_{2n+2}^{(2)}.$$ 
Note that $v_{1}^{(1)} \in \mathcal{A}_{1,1}, v_{1}^{(2)} \in \mathcal{A}_{1,2}, v_{2n+2}^{(1)} \in \mathcal{A}_{2n+2,1}$ and $v_{2n+2}^{(2)} \in \mathcal{A}_{2n+2,2}.$ Each of these correspond to distinct eigenspaces of $T^*T$. Since $h_{n,i}$ is non-zero, so is $T^{n+2}h_{n,i}.$ Therefore one of the coefficients of this sum must be non zero. This implies that one of $v_{1}^{(1)}, v_{1}^{(2)}, v_{2n+2}^{(1)}$ or $v_{2n+2}^{(2)}$ is in $\mathcal{K}$.  It follows that $H(n) \cap \mathcal{K} \neq \emptyset$ for some $n.$

\item[$\boldsymbol{\lambda=1}$]
 A similar calculation as in the case of $\lambda=0$ ensures the existence of some  $n$ with $\mathcal{K} \cap H(n) \neq \emptyset$.
\end{description}

These three cases ensure the existence of an $n$ such that $\mathcal{K} \cap H(n) \neq \emptyset$. Since each $T_{n}$ is invertible, by applying $T^{k}$ for sufficiently large $k$ it follows that there exists $m > 0$ such that $\mathcal{K} \cap H(m) \neq \emptyset.$ Pick a non-zero element $h_{m}$ from $\mathcal{K} \cap H(m)$. Then $h_{m} = \alpha v_{m}^{(1)} + \beta v_{m}^{(2)} = \gamma u_{m}^{(1)} + \delta u_{m}^{(2)}$. We have already shown that $A_{m}B_{m} - B_{m}A_{m} \neq 0$, therefore either $\alpha \beta \neq 0$ or $\gamma \delta \neq 0$.   If $\alpha \beta \neq 0$, then $v_{m}^{(1)}, v_{m}^{(2)} \in \mathcal{K}$ since $v_{m}^{(1)}, v_{m}^{(2)}$ are in different eigenspaces of $T^{*}T.$ Similarly,  $u_{m}^{(1)}, u_{m}^{(2)} \in \mathcal{K}$ if $\gamma \delta \neq 0$. We conclude that  $H(m) \subseteq \mathcal{K}$. Now since $T_{n}$ is invertible for all $n$,  applying $T^{n}$ and ${T^{*}}^{n}$ on $H(m),$ we find that $H(k) \subseteq \mathcal{K}$ for all $k$. This implies that $\mathcal{K} = L^{2}(\mathbb{T}) \oplus L^{2}(\mathbb{T})$ completing the proof. 
\end{proof}

Let $B(\lambda, s, \alpha)$ denote the operator $\left[\begin{array}{ccc}
                    B(s) & \alpha (B(s) - B)\\
                      0 & B
                  \end{array}\right].$ Now we show that the unitary equivalence class of $B(\lambda, s, \alpha)$ depends only on $\lambda, |a|,$ (where $s = ia$) and $| \alpha |.$

\begin{thm}\label{inequivalance within principal series rep}
The operators $B(\lambda_{1}, s_{1}, \alpha_{1})$ and $B(\lambda_{2}, s_{2}, \alpha_{2})$ are unitarily equivalent if and only if $\lambda_{1} = \lambda_{2},$ $a_{1} = a_{2}$ and $\alpha_{1} = \alpha_{2}$ for any choice of a pair of purely imaginary numbers $s_{1} = ia_{1}, s_{2} = ia_{2},$ $a_{1}, a_{2} > 0,$ and $\alpha_{1}, \alpha_{2}>0$.
\end{thm}

\begin{proof} The operators $B(\lambda_{i}, s_{i}, \alpha_{i})$ are homogeneous with associated representation  $P_{\lambda_{i}, s_{i}} \oplus P_{\lambda_{i}, s_{i}}$ for $i = 1, 2,$  see Corollary \ref{homog S}.  If $\lambda_{1} \neq \lambda_{2}$ then the multipliers of $P_{\lambda_{1}, s_{1}} \oplus P_{\lambda_{1}, s_{1}}$ and $P_{\lambda_{2}, s_{2}} \oplus P_{\lambda_{2}, s_{2}}$ are inequivalent \cite[Corollary 3.2]{Homogshift}. 
Therefore $P_{\lambda_{1}, s_{1}} \oplus P_{\lambda_{1}, s_{1}}$ and $P_{\lambda_{2}, s_{2}} \oplus P_{\lambda_{2}, s_{2}}$ are inequivalent. Since the representation associated with an irreducible homogeneous operator is uniquely determined, it follows that $B(\lambda_{1}, s_{1}, \alpha_{1})$ and $B(\lambda_{2}, s_{2}, \alpha_{2})$ cannot be inequivalent and consequently, $\lambda_1 = \lambda_2.$ 

Now, setting $\lambda_{1} = \lambda_{2} = \lambda,$ we show that if $B(\lambda, s_{1}, \alpha_{1})$ and $B(\lambda, s_{2}, \alpha_{2})$ are equivalent, then $s_{1} = s_{2}$ and $\alpha_{1} = \alpha_{2}$. 

Since $B(\lambda, s_{1}, \alpha_{1})$ and $B(\lambda, s_{2}, \alpha_{2})$ are equivalent, it follows that the set of singular values of these two operators must be the same and consequently the two sets 
$$\mathcal S_1:=\left\lbrace \alpha_{1}^2 \frac{|2s_{1}|^{2}}{|n + \frac{1+\lambda}{2} - s_{1}|^{2}} : n \in \mathbb{Z} \right\rbrace\,\, \mbox{\rm and}\,\,  \mathcal S_2:=\left\lbrace \alpha_{2}^2 \frac{|2s_{2}|^{2}}{|n + \frac{1+\lambda}{2} - s_{2}|^{2}} : n \in \mathbb{Z} \right\rbrace$$ 
must be the same.
\begin{description}
\item[$\boldsymbol{\lambda < 0}$]
In this case the maximum  of the sets $\mathcal{S}_1$ and $\mathcal{S}_2,$ which is achieved at $n=0$ in both cases, 
must be equal, that is, 
\begin{equation}\label{eqn:2.1.2}
\frac{4\alpha_{1}^2 a_{1}^{2}}{(\frac{1+\lambda}{2})^{2} + a_{1}^{2}} = \frac{4\alpha_{2}^2 a_{2}^{2}}{(\frac{1+\lambda}{2})^{2} + a_{2}^{2}}.
\end{equation}
Removing this maximum from both $\mathcal{S}_1$ and $\mathcal{S}_2,$ again, the maximum in each of them is achieved at $n=-1$  and they must be equal, that is, 
\begin{equation}\label{eqn:2.1.3}
\frac{4\alpha_{1}^2 a_{1}^{2}}{(-1 + \frac{1+\lambda}{2})^{2} + a_{1}^{2}} = \frac{4\alpha_{2}^2 a_{2}^{2}}{(-1 + \frac{1+\lambda}{2})^{2} + a_{2}^{2}}.
\end{equation}
Combining equations \eqref{eqn:2.1.2} and \eqref{eqn:2.1.3}, we obtain the equation
$\alpha_{1}^2 a_{1}^{2} =  \alpha_{2}^2 a_{2}^{2}.$
Using this relationship in \eqref{eqn:2.1.2}, we find that $a_{1}^{2} = a_{2}^{2}$. Since both $a_{1}$ and $a_{2}$ are positive, it follows that $a_{1} = a_{2}.$ Therefore $\alpha_{1} = \alpha_{2}$.

\item[$\boldsymbol{\lambda = 0}$] As before, in this case, the  maximum and the second maximum value of $\mathcal{S}_1$ and $\mathcal S_2$ are achieved at  $n = 0$ and $n=1,$ respectively. So, equating these two values, we get  
\begin{equation*}
\frac{4\alpha_{1}^2 a_{1}^{2}}{\frac{1}{4}+ a_{1}^{2}} = \frac{4\alpha_{2}^2 a_{2}^{2}}{\frac{1}{4} + a_{2}^{2}}\,\, \,\,\mbox{\rm and}\,\,\,\,\frac{4\alpha_{1}^2 a_{1}^{2}}{\frac{9}{4} + a_{1}^{2}} = \frac{4\alpha_{2}^2 a_{2}^{2}}{\frac{9}{4} + a_{2}^{2}}.
\end{equation*}
A similar calculation, as in the case of $\lambda <0,$ implies that $a_{1} = a_{1}$ and $\alpha_{1} = \alpha_{2}$. 

\item[$\boldsymbol{\lambda > 0}$] One last time,  we note that the maximum and the second maximum of the two sets $\mathcal{S}_1$ and $\mathcal S_2$ are achieved at $n = -1$ and $n = 0,$ respectively. Equating these values, we obtain a pair of equations identical to the equations we had obtained  in the case of $\lambda < 0$. 
Therefore we conclude that $a_{1} = a_{2}$ and $\alpha_{1} = \alpha_{2}$.
\end{description}
Thus $B(\lambda_{1}, s_{1}, a_{1})$ and $B(\lambda_{2}, s_{2}, a_{2})$ are equivalent if and only if $\lambda_{1} = \lambda_{2}$, $a_{1} = a_{2}$ and $\alpha_{1} = \alpha_{2}$.
\end{proof}
If  $U_{\lambda, s} : L^{2}(\mathbb{T}) \rightarrow L^{2}(\mathbb{T})$ is the operator 
$U_{\lambda, s} z^{n} = \frac{\Gamma (n+\frac{1+\lambda}{2} - s)}{\Gamma (n+\frac{1+\lambda}{2} + s)} z^{n},$
then from \cite[p. 318]{Homogshift}, it follows that 
$U_{\lambda, s}$ is  unitary, $U_{\lambda, -s} = U_{\lambda, s}^{*},$   $P_{\lambda, -s}U_{\lambda, s} = U_{\lambda, s} P_{\lambda, s}$ and $B(s) = U_{\lambda, s}^{*} B U_{\lambda, s}$. 

Replacing $s$ by $-s,$ we see that $B(-s) = U_{\lambda, -s}^{*} B U_{\lambda, -s}.$ This is the same as $B(-s) = U_{\lambda, s} B U_{\lambda,s}^{*}$. Consequently,  $U_{\lambda, s} \oplus U_{\lambda, s}$ intertwines $\left[\begin{smallmatrix}
                    B(s) & \alpha (B(s) - B)\\
                      0 & B
                  \end{smallmatrix}\right]$ and $\left[\begin{smallmatrix}
                    B & \alpha (B - B(-s))\\
                      0 & B(-s)
                  \end{smallmatrix}\right].$ 
It follows, after conjugating with a permutation,  that $\left[\begin{smallmatrix}
                    B(s) & \alpha (B(s) - B)\\
                      0 & B
                  \end{smallmatrix}\right]$ and $\left[\begin{smallmatrix}
                    B(-s) & \alpha (B(-s) - B)\\
                      0 & B
                  \end{smallmatrix}\right]$ are unitarily equivalent.  
%
Hence $$\mathcal P:=\left \lbrace 
                  \left[ \begin{array}{ccc}
                    B(s) & \alpha (B(s) - B)\\
                      0 & B
                  \end{array}\right] : \lambda, s=ia, a > 0, \alpha > 0\right \rbrace$$ is a mutually unitarily inequivalent set of irreducible homogeneous operators with associated representation $P_{\lambda, s} \oplus P_{\lambda, s}$.
                  
The associated representation of the family of irreducible homogeneous operators 
$$\mathcal C:=\left \lbrace T(a,b, \alpha) = 
                  \left[\begin{array}{ccc}
                    T(a,b) & \alpha (T(a,b) - T(b,a))\\
                      0 & T(b,a)
                  \end{array}\right] : 0 < a < b < 1, \alpha > 0 \right \rbrace$$ 
                  is the direct sum of two copies of a Complementary series representation (see \cite{HomogKoranyi}). 
                  We now show that these  two sets of homogeneous operators are mutually unitarily inequivalent. 
                  
\begin{thm}\label{Inequivalance between complementary series and principal series}                                   
The homogeneous operators in the two sets $\mathcal P$ and $\mathcal C$  are mutually unitarily inequivalent.
\end{thm}

\begin{proof}
Let $T(a,b,\alpha)$ and $B(\lambda_{1}, s, \beta)$ be unitarily  equivalent for some 
\begin{enumerate}
\item[] $(a,b,\alpha) : 0 < a < b < 1$, $\alpha > 0;$ 
\item [] $(\lambda_1, \beta, s): -1 < \lambda_1 \leq 1,$ $\beta > 0$ and  
$s,$ $k=$Im$(s) > 0.$ 
\end{enumerate}
The associated representation of the operator $T(a,b,\alpha)$ is $C_{\lambda, \sigma} \oplus C_{\lambda, \sigma},$ where $\lambda = a+b-1,$ $\sigma=\tfrac{b-a}{2}$ (cf. \cite[Lemma 2.1]{HomogKoranyi}) and the associated representation of $B(\lambda_{1}, s, \beta)$ is $P_{\lambda_{1}, s} \oplus P_{\lambda_{1}, s}$, see Corollary \ref{homog S}. Since the representation associated with an irreducible homogeneous operator is uniquely determined, it follows that 
$C_{\lambda, \sigma} \oplus C_{\lambda, \sigma}$ and  $P_{\lambda_{1}, s} \oplus P_{\lambda_{1}, s}$ must be equivalent. This, in particular, implies that 
their multipliers are equivalent and, therefore, $\lambda_1$ must be equal to $\lambda.$  For the remaining portion of the proof, we therefore assume that $\lambda_1=\lambda$ without loss of generality.



From \cite[Equation 2.6]{HomogKoranyi}, it follows that if the $2$-shifts $T(a,b,\alpha)$ and $B(\lambda, s, \beta)$ are unitarily equivalent, then the two sets 
$$\mathcal{S}_1 = \left\lbrace \frac{(1+\alpha^{2})(a-b)^{2}}{\left(n + \frac{1+\lambda}{2} \right)^{2} - \left( \frac{a-b}{2}\right)^{2}} : n \in \mathbb{Z} \right\rbrace\,\,\mbox{and}\,\,\mathcal{S}_2 = \left\lbrace \frac{4\beta^2 k^{2}}{\left(n + \frac{1+\lambda}{2} \right)^{2} + k^{2}} : n \in \mathbb{Z} \right\rbrace$$
must be equal. Suppose ${\lambda < 0}.$ Then following the same analysis as in the proof of Theorem \ref{inequivalance within principal series rep}, we arrive at a contradiction. Similarly, if ${\lambda > 0}$ or ${\lambda = 0}$, we arrive at a contradiction. 
It follows that $T(a,b,\alpha)$ is not equivalent to $B(\lambda_{1}, s, \beta)$  for any choice of $(a,b,\alpha)$ and $(\lambda_1, s, \beta).$
\end{proof}

\subsubsection{ \bf The case of ``$\,\,\mathbf {s= 0}$'':}  Having disposed of the case of $s\not =0,$ in what follows, we assume $s = 0$ with one  exception in the Proposition below.

\begin{prop}\label{Finding S when s=0}
Suppose $S$ is  an operator on $L^{2}(\mathbb{T})$ such that  
\begin{equation}\label{eqn:2.1.6}
S P_{\lambda, s}(\phi) - e^{i \theta} P_{\lambda, s}(\phi) S = \overline{a} B P_{\lambda, s}(\phi) S + \overline{a} S P_{\lambda, s}(\phi) B
\end{equation}
for all $\phi$ in \mbox{\rm M\"{o}b}. Then 
\begin{enumerate}
\item[]
$(a)$ if $s \neq 0$, then $S = 0$ and  
\item[]
$(b)$ if $s = 0$ and $\lambda \neq 1$, then $S$ is a weighted shift operator on $L^2(\mathbb{T})$ with respect to the orthonormal basis $\{z^n\}$ with weight sequence $\left\lbrace \alpha_n = \frac{1}{\lambda + 2n - 1} \right\rbrace$.
\end{enumerate}
\end{prop}
\begin{proof}
From Appendix A(I), it follows that $S$ is a weighted shift operator with respect to the orthonormal basis $\{z^{n}\}$ in $L^2(\mathbb T)$. 
Let $\{\alpha_{n}\}$ be the weight sequence of $S$. 

\begin{description}
\item[Proof of $\boldsymbol{(a)}$] Substituting $m = n-1$ in the equation \eqref{basic eq: 4} 
and comparing the coefficient of $r^{k},$  $k\geq 1,$ we obtain
\begin{equation*}
\alpha_{n-1}C_{k}(n-1, n-1) - \alpha_{n-1}C_{k}(n,n) = 2 \alpha_{n-1} C_{k}(n,n-1)
\end{equation*}
and it follows that 
$2s\alpha_{n-1}= 0.$
Therefore, if $s \neq 0$, then  $\alpha_{n-1} = 0$ for all $n \in \mathbb{Z}$. This completes the proof of $(a)$.

\item[Proof of $\boldsymbol{(b)}$]  Putting $m = n$ in the equation \eqref{basic eq: 4}
and comparing the coefficient of $r,$  we have 

\begin{equation}\label{eqn:2.1.9}
\alpha_{n-1}(\lambda + 2n - 1) = \alpha_{n}(\lambda + 2n + 1).
\end{equation}
Evidently,  $\alpha_{n} = \frac{1}{\lambda + 2n + 1}$ is a solution to the recursion  \eqref{eqn:2.1.9}.

Equating the coefficient of $r^{k}$  in the equation \eqref{basic eq: 4} 
we obtain 
\begin{equation*}
\alpha_{n-1}(\lambda + 2n - 1) = \alpha_{m} (\lambda + 2m + 1).
\end{equation*}
Thus $\alpha_{n} = \frac{1}{\lambda + 2n + 1}$ is a solution of  this recursion as well. 
%
This shows that if $\alpha_{n} = \frac{1}{\lambda + 2n + 1}$, $n \in \mathbb{Z},$ then $S$ satisfies equation \eqref{eqn:2.1.6} for any involution $\phi_a$,  $a \in \mathbb{D}$. Also $S$ satisfies equation \eqref{eqn:2.1.6} for the subgroup of rotations $\phi_\theta$. Since any elements of M\"{o}b is composition of $\phi_\theta$ and $\phi_a$ for some $\theta$ and $a$, it follows that $S$ satisfies equation \eqref{eqn:2.1.6} for every elements of M\"{o}b. This completes the proof of part $(b)$.
\end{description}\vskip -1.5em\end{proof}

Let  $S{[\lambda]}$ be the weighted shift operator on $L^2(\mathbb{T})$ with respect to the orthonormal basis $\{z^n\}$ with weight sequence $\left\lbrace \frac{1}{\lambda + 2n - 1} \right\rbrace.$ Also, let 
$$B(\lambda, \alpha): L^2(\mathbb T) \oplus L^2(\mathbb T) \to L^2(\mathbb T) \oplus  L^2(\mathbb T),\,\,  B(\lambda, \alpha):=\left[ \begin{array}{ccc}
                    B & \alpha S[\lambda]\\
                      0 & B
                  \end{array}\right],\,\, \alpha \in \mathbb C.$$  
\begin{cor}\cite[Lemma 2.1]{HomogKoranyi}\label{homog S with s=0}
The operator $B(\lambda,\alpha)$ is homogeneous with associated representation 
                  $P_{\lambda,0} \oplus P_{\lambda,0}$ with $\lambda \neq 1.$
\end{cor}
The proof of the following theorem is similar to that of Theorem \ref{Irrducibility of operators corresponding to principal series rep} and therefore omitted.               
\begin{thm}\label{irreducibility of homog op with associated rep P_lambda, 0}
For every fixed but arbitrary $\lambda, \alpha$ with $-1 < \lambda < 1$ and $\alpha > 0$, the operator
$B(\lambda,\alpha)$ is irreducible.
\end{thm}


Now we have another class of irreducible homogeneous operators, $\left\lbrace B(\lambda, \alpha) : -1 < \lambda < 1, \alpha \right\rbrace,$
whose associated representation is $P_{\lambda, 0} \oplus P_{\lambda, 0}$. Clearly, $B(\lambda, \alpha)$ and $B(\lambda, |\alpha|)$ are unitarily equivalent. 
However, part (a) of the following theorem says that the  irreducible homogeneous operators in the set 
$$\mathcal{P}_0 = \left\lbrace B(\lambda, \alpha) : -1 < \lambda < 1, \alpha > 0 \right\rbrace$$
are mutually unitarily inequivalent.


%

\begin{thm}\label{inequivalance between the class of homog op with associated rep P_lambda, s and P_lambda, 0}
\begin{enumerate}\item[(a)\phantom{444}] \hskip -1.5em Let $\alpha_{1}, \alpha_{2} > 0$ and $-1 < \lambda_1, \lambda_2 < 1$. The operators $B(\lambda_{1}, \alpha_{1})$ and $B(\lambda_{2}, \alpha_{2})$ are unitarily equivalent if and only if $\lambda_{1} = \lambda_{2}$ and $\alpha_{1} = \alpha_{2}$.
\item[(b)] The homogeneous operators in the two sets $\mathcal{C}, \mathcal P$ and $\mathcal P_0$  are mutually unitarily inequivalent.

\item[(c)] The homogeneous operators in the two sets $\mathcal C$ and $\mathcal P_0$  are mutually unitary inequivalent.
\end{enumerate}
\end{thm}
\begin{proof}
The proof of the statement in (a) is similar to that of Theorem \ref{inequivalance within principal series rep}.
The proof of the statement in (b)  is  similar to that of  Theorem \ref{inequivalance within principal series rep} and 
the proof of the statement in (c) is similar to that of Theorem \ref{Inequivalance between complementary series and principal series}.
\end{proof}

\subsection{Classification}
\begin{thm}\label{homog op when associated rep is direct sum of cont. series rep}
Let $\pi_{1} = R_{\lambda_{1}, \mu_{1}}$ and $\pi_{2} = R_{\lambda_{2}, \mu_{2}}$ be two representations from the Continuous series, excluding $P_{1,0}$, acting on the Hilbert spaces  $H_1$ and $H_2,$ respectively. Assume that $(\lambda_{1}, \mu_{1}) \neq (\lambda_{2}, \mu_{2})$. Suppose 
$$T = \left[\begin{array}{ccc}
    T_{1} & S_{1}\\
    S_{2} & T_{2}
\end{array}\right]$$
is a homogeneous operator on $H = H_1 \oplus H_2$ with associated representation $\pi_{1} \oplus \pi_{2}$. Then either $S_{1} = 0$ or $S_{2} = 0$. Furthermore, $S_{1} = 0$ and $S_{2} = 0$ except when $R_{\lambda_{1}, \mu_{1}} = P_{\lambda, s}$ and $R_{\lambda_{2}, \mu_{2}} = P_{\lambda, -s}$.
\end{thm}

\begin{proof}
Since $T$ is a homogeneous operator with associated representation $\pi_{1} \oplus \pi_{2}$, we have $$\phi(T) = (\pi_{1}(\phi)^{*} \oplus \pi_{2}(\phi)^{*}) T (\pi_{1}(\phi) \oplus \pi_{2}(\phi)),\,\, \phi \in \mbox{\rm M\"{o}b}.$$ For $\phi_{\theta, a}$ in  M\"{o}b, this is equivalent to the four equations  listed below:
\begin{equation}\label{eqn:2.1.17}
e^{i \theta} \pi_{1}(\phi_{\theta, a}) (T_{1} - aI) = T_{1} \pi_{1}(\phi_{\theta, a}) (I - \overline{a}T_{1}) - \overline{a} S_{1} \pi_{2}(\phi_{\theta, a}) S_{2}
\end{equation}
\begin{equation}\label{eqn:2.1.18}
S_{1}\pi_{2}(\phi_{\theta, a}) - e^{i \theta} \pi_{1}(\phi_{\theta, a}) S_{1} = \overline{a}T_{1}\pi_{1}(\phi_{\theta, a})S_{1} + \overline{a} S_{1} \pi_{2}(\phi_{\theta, a}) T_{2}
\end{equation}
\begin{equation}\label{eqn:2.1.19}
S_{2}\pi_{1}(\phi_{\theta, a}) - e^{i \theta} \pi_{2}(\phi_{\theta, a}) S_{2} = \overline{a}S_{2}\pi_{1}(\phi_{\theta, a})T_{1} + \overline{a} T_{2} \pi_{2}(\phi_{\theta, a}) S_{2}
\end{equation}
\begin{equation}\label{eqn:2.1.20}
e^{i \theta} \pi_{2}(\phi_{\theta, a}) (T_{2} - aI) = T_{2} \pi_{2}(\phi_{\theta, a}) (I - \overline{a}T_{2}) - \overline{a} S_{2} \pi_{1}(\phi_{\theta, a}) S_{1}.
\end{equation}
Evaluating equation \eqref{eqn:2.1.17} on $z^{n}$ for $\phi_{\theta},$ we get  
\begin{equation*}
\pi_{1}(\phi_{\theta}) T_{1}z^{n} = e^{-i\left(n+1+\frac{\lambda_{1}}{2}\right) \theta} T_{1} z^{n}.
\end{equation*}
This proves that $T_{1}$ is a weighted shift operator with respect to the orthonormal basis $\left\lbrace \frac{z^{n}}{\|z^{n}\|_{1}} \right\rbrace$, where $\|\cdot\|_i$ denote the norm of $H_i$. 
Let $\{u_{n}\}$ be the weight sequence of $T_{1}$.

Similarly, it may be  shown that $T_{2}$ is a weighted shift operator with respect to the orthonormal basis $\left\lbrace \frac{z^{n}}{\|z^{n}\|_{2}} \right\rbrace$. Let $\{v_{n}\}$ be the weight sequence of $T_{2}$.

If $\lambda_{1} \neq \lambda_{2}$, then the equation \eqref{basic eq: 3} implies that $S_{1}z^{n} = 0,$ $n \in \mathbb{Z}.$ Consequently, $S_{1} = 0$. Similarly, it can be shown that $S_{2} = 0,$ whenever $\lambda_{1} \neq \lambda_{2}.$ Therefore, the proof, in this case, is complete and we may assume, without loss of generality, that $\lambda:=\lambda_1=\lambda_2.$

The existence of a sequence $\{\alpha_{n} : n \in \mathbb{Z} \}$ such $S_{1} e_{n}^{2} = \alpha_{n} e_{n+1}^{1},$ where $e_{n}^{i} = \frac{z^{n}}{\|z^{n}\|_{i}}$,  $i = 1, 2,$ follows from the equation \eqref{basic eq: 3} in the Appendix.

In the equation \eqref{basic eq: 4}, putting $m = n-1,$ then differentiating with respect to $r$ and finally substituting $r = 0,$ we get 
\begin{multline*}
v_{n-1}\alpha_{n-1} \frac{\|z^{n-1}\|_{2}}{\|z^{n}\|_{2}} (- \mu_{2} + n) + \alpha_{n-1} u_{n-1} \frac{\|z^{n-1}\|_{1}}{\|z^{n}\|_{1}} (- \mu_{1} + n)\\
= \alpha_{n-1} [(\mu_{2} - \mu_{1})(\lambda + \mu_{2} + \mu_{1} - 1) + (\lambda + 2n - 1)].
\end{multline*}
It follows that if $\alpha_{n-1} \neq 0$,  then  
\begin{equation}\label{eqn:2.1.22}
v_{n-1} \tfrac{\|z^{n-1}\|_{2}}{\|z^{n}\|_{2}} (- \mu_{2} + n) + u_{n-1} \tfrac{\|z^{n-1}\|_{1}}{\|z^{n}\|_{1}} (-\mu_{1} + n)
= (\mu_{2} - \mu_{1})(\lambda + \mu_{2} + \mu_{1} - 1) + (\lambda + 2n - 1), n \in \mathbb{Z}.
\end{equation}
The existence of a sequence $\{ \beta_{n} \}$ such that $S_{2} e_{n}^{1} = \beta_{n} e_{n+1}^{2}$,  $n \in \mathbb{Z},$ follows from a similar computation. As before, for the sequence  $\beta_{n-1},$ we also have 
\begin{multline*}
v_{n-1}\beta_{n-1} \frac{\|z^{n-1}\|_{2}}{\|z^{n}\|_{2}} (- \mu_{2} + n) + \beta_{n-1} u_{n-1} \frac{\|z^{n-1}\|_{1}}{\|z^{n}\|_{1}} (-\mu_{1} + n)\\
= \beta_{n-1} [(\mu_{1} - \mu_{2})(\lambda + \mu_{2} + \mu_{1} - 1) + (\lambda + 2n - 1)].
\end{multline*}
Thus if $\beta_{n-1} \neq 0$, then we have
\begin{equation}\label{eqn:2.1.24}
v_{n-1} \tfrac{\|z^{n-1}\|_{2}}{\|z^{n}\|_{2}} (- \mu_{2} + n) + u_{n-1} \tfrac{\|z^{n-1}\|_{1}}{\|z^{n}\|_{1}} (-\mu_{1} + n)
= (\mu_{1} - \mu_{2})(\lambda + \mu_{2} + \mu_{1} - 1) + (\lambda + 2n - 1), n \in \mathbb{Z}.
\end{equation}
Equating the right hand sides of equations \eqref{eqn:2.1.22} and \eqref{eqn:2.1.24}, we get $\mu_{1} = \mu_{2},$ contradicting our hypothesis that $\mu_1 \neq \mu_2$. Therefore, we can find an integer $n$ such that either $\alpha_{n-1} = 0$ or $\beta_{n-1} = 0$. Assume that $\alpha_{p} = 0$, for some integer $p$. 

Now putting $m = n$ in the equation \eqref{basic eq: 4}, then differentiating with respect to $r$  and lastly putting $r = 0,$ we get
\begin{equation*}
\alpha_{n-1} \frac{\|z^{n-1}\|_{2}}{\|z^{n}\|_{2}} (-\mu_{2} + n) = \alpha_{n} \frac{\|z^{n}\|_{1}}{\|z^{n+1}\|_{1}} (-\mu_{1} + n).
\end{equation*}
In this recursion, for  all $n \in \mathbb{Z},$ the coefficients of $\alpha_{n-1},$ $\alpha_n$ are non zero. Thus if $\alpha_{p} = 0$ for some integer $p$, then $\alpha_n=0$ for all $n\in \mathbb Z$ and consequently, $S_{1} = 0$. 

Similarly, if $\beta_{p} = 0$ for some $p,$ then $S_{2} = 0$. This completes the proof of the first part of the Theorem.

Now assume that $S_{2} = 0$. Then \cite[Proposition 2.4]{CONSTCharGMBB} implies that $T_{i}$'s are homogeneous operators with associated representation $\pi_{i}$. Since all the homogeneous shifts are known, the weights of $T_1$ and $T_2$ are therefore known. 

Suppose $S_1 \neq 0$. Then one of the weights of $S_1$ must be non-zero. Choose, without loss of generality,  $\alpha_{n-1} \neq 0$ for some $n\in\mathbb Z.$ For this choice of $\alpha_{n-1},$ we have equation \eqref{eqn:2.1.22}.  
\begin{enumerate}
\item[(a)] Assume that both $\pi_{1}$ and $\pi_2$ are from the Complementary series, that is, $\pi_{i} = C_{\lambda, \sigma_{i}},$ where $0 < \sigma_{i} < \frac{1}{2}(1 - |\lambda|)$ and $\mu_{i} = \frac{1 - \lambda}{2} + \sigma_{i},$ $i=1,2.$  In this case, we have $\frac{\|z^{n-1}\|_{i}^{2}}{\|z^{n}\|_{i}^{2}} = \frac{\lambda + \mu_{i} + n - 1}{- \mu_{i} + n},$ $i=1,2.$
Since $T_i$ are homogeneous operators with associated representation $C_{\lambda, \sigma_{i}},$ ${T_i^{-1}}^*$ is also homogeneous with the same associated representation and we have the following possibilities for the weight sequences.

\begin{enumerate}\item[(i)] For  $n \in \mathbb{Z},$ assume that $u_{n-1} = \frac{\|z^{n}\|_{1}}{\|z^{n-1}\|_{1}}$ and $v_{n-1} = \frac{\|z^{n}\|_{2}}{\|z^{n-1}\|_{2}}$. Then from the equation \eqref{eqn:2.1.22}, we obtain $(\sigma_{2} + \sigma_{1})(\sigma_{2} - \sigma_{1} + 1) = 0,$ which is a contradiction. 

\item[(ii)] For all $n \in \mathbb{Z},$ assume that $u_{n-1} = \frac{\|z^{n-1}\|_{1}}{\|z^{n}\|_{1}}$ and $v_{n-1} = \frac{\|z^{n}\|_{2}}{\|z^{n-1}\|_{2}}$. Then from the equation \eqref{eqn:2.1.22}, we obtain $(\mu_{2} - \mu_{1})(\sigma_{1} + \sigma_{2} + 1) = 0,$ which is a contradiction.

\item[(iii)] For all $n \in \mathbb{Z},$ assume that $u_{n-1} = \frac{\|z^{n}\|_{1}}{\|z^{n-1}\|_{1}}$ and $v_{n-1} = \frac{\|z^{n-1}\|_{2}}{\|z^{n}\|_{2}}$. Then from the equation \eqref{eqn:2.1.22}, we get $(\sigma_{2} - \sigma_{1})(\sigma_{1} + \sigma_{2} - 1) = 0,$ which is a contradiction.


\item[(iv)] For all $n \in \mathbb{Z},$ assume that $u_{n-1} = \frac{\|z^{n-1}\|_{1}}{\|z^{n}\|_{1}}$ and $v_{n-1} = \frac{\|z^{n-1}\|_{2}}{\|z^{n}\|_{2}}$. Now from the equation \eqref{eqn:2.1.22}, we get $\sigma_{2} - \sigma_{1} = 1,$ which is a contradiction.

\end{enumerate}

Combining (i) -- (iv), we find that there does not exists any $n$ for which $\alpha_{n-1} \neq 0$ and we conclude that $S_{1} = 0$ in this case.

\item[(b)] Let $\pi_{1} = C_{\lambda, \sigma}$ for some $0 < \sigma < \frac{1}{2}(1 - |\lambda|)$ and $\pi_{2} = P_{\lambda, s}$ where $s$ is purely imaginary. Now $\mu_{1} = \frac{1 - \lambda}{2} + \sigma$ and $\mu_{2} = \frac{1-\lambda}{2} + s$. Since the representation space of $\pi_{2}$ is $L^{2}(\mathbb{T})$, $\|z^{n}\|_{2} = 1$ for all $n \in \mathbb{Z}$. 

Recall that there are two homogeneous operators whose associated representation is $\pi_{2}$, one is the unweighted bilateral shift and the other one is the weighted shift with weight sequence $\left\lbrace v_{n -1} = \frac{-\mu_{2} + n + 2s}{-\mu_{2} + n} \right\rbrace$. As before, we consider four different possibilities that arise in this case. In each of these cases, a contradiction is obtained by noting that $s$ is purely imaginary.
\begin{enumerate}
\item[(i)] For all $n \in \mathbb{Z},$ assume that $v_{n-1} = 1$ and $u_{n-1} = \frac{\|z^{n}\|_{1}}{\|z^{n-1}\|_{1}}$. Substituting these values of $u_{n-1}$ and $v_{n-1}$ in the equation \eqref{eqn:2.1.22}, we get $s^{2} - \sigma^{2} + \sigma + s = 0$. 

\item[(ii)] For all $n \in \mathbb{Z},$ assume that $v_{n-1} = 1$ and $u_{n-1} = \frac{\|z^{n-1}\|_{1}}{\|z^{n}\|_{1}}$. Substituting these values of $u_{n-1}$ and $v_{n-1}$ in the equation \eqref{eqn:2.1.22}, we get $(s - \sigma) (s + \sigma + 1) = 0$. 

\item[(iii)] For all $n \in \mathbb{Z},$ assume that $v_{n-1} = \frac{-\mu_{2} + n + 2s}{-\mu_{2} + n}$ and $u_{n-1} = \frac{\|z^{n}\|_{1}}{\|z^{n-1}\|_{1}}$. Substituting these values of $u_{n-1}$ and $v_{n-1}$ in the equation \eqref{eqn:2.1.22}, we get $s^{2} - \sigma^{2} + \sigma - s = 0$. 

\item[(iv)] For all $n \in \mathbb{Z},$ assume that $v_{n-1} = \frac{-\mu_{2} + n + 2s}{-\mu_{2} + n}$ and $u_{n-1} = \frac{\|z^{n-1}\|_{1}}{\|z^{n}\|_{1}}$. Substituting these values of $u_{n-1}$ and $v_{n-1}$ in the equation \eqref{eqn:2.1.22}, we get $s^{2} - \sigma^{2} - \sigma - s = 0$. 
\end{enumerate}
Combining (i) - (iv), again in this case, we see that $S_{1} = 0$.

\item[(c)] For $i=1,2,$ assume that $\pi_{i}=  P_{\lambda, s_{i}}$ are two  Principal series representations. We have the following four cases to consider. 
\begin{enumerate}
\item[(i)]  For all $n \in \mathbb{Z},$ assume that $u_{n-1} = \frac{-\mu_{1} + n + 2s_{1}}{- \mu_{1} + n}$ and $v_{n-1} = 1$. Substituting these values of $u_{n-1}$ and $v_{n-1}$ in the equation \eqref{eqn:2.1.22}, we get $(s_{2} - s_{1})(s_{2} + s_{1} + 1) = 0$. This is a contradiction since $s_{1} \neq s_{2}$.

\item[(ii)] For all $n \in \mathbb{Z},$ assume that $u_{n-1} = 1$ and $v_{n-1} = \frac{-\mu_{2} + n + 2s_{2}}{- \mu_{2} + n}$. Substituting these values of $u_{n-1}$ and $v_{n-1}$ in the equation \eqref{eqn:2.1.22}, we get $s_{2}^{2} - s_{1}^{2} + s_{1} - s_{2} = 0$. This is a contradiction since $s_{1} \neq s_{2}$.

\item[(iii)] For all $n \in \mathbb{Z},$ assume that $u_{n-1} = \frac{-\mu_{1} + n + 2s_{1}}{- \mu_{1} + n}$ and $v_{n-1} = \frac{-\mu_{2} + n + 2s_{2}}{- \mu_{2} + n}$. Substituting these values of $u_{n-1}$ and $v_{n-1}$ in the equation \eqref{eqn:2.1.22}, we get $s_{2}^{2} - s_{1}^{2} = s_{2} + s_{1}$. Since $s_{1} \neq s_{2}$ and both of them are purely imaginary, it follows, from the preceding equation, that $s_{2} = -s_{1}$. 


\item[(iv)] For all $n \in \mathbb{Z},$ assume that $u_{n-1} = 1$ and $v_{n-1} =1$. Substituting these values of $u_{n-1}$ and $v_{n-1}$ in the equation \eqref{eqn:2.1.22}, we get $s_{2}^{2} - s_{1}^{2} + s_{2} + s_{1} = 0$. We conclude that  $s_{2} = -s_{1}$ exactly as before.
\end{enumerate}
\end{enumerate}
The proof is complete by putting together the results of the three cases {(a) - (c)}. 
\end{proof}

\begin{prop}\label{when s not 0, if diagonals are B(s) then S = 0}
Let $P_{\lambda, s}$ be a  representation from the Principal series with $s \neq 0$.  If  $S$ is any  operator on $L^{2}(\mathbb{T})$ such that 
\begin{equation}\label{eqn:2.1.10}
SP_{\lambda, s}(\phi) - e^{i \theta} P_{\lambda, s}(\phi) S = \overline{a} B(s) P_{\lambda, s}(\phi) S + \overline{a} S P_{\lambda, s}(\phi) B(s),\,\, \phi \in \mbox{M\"{o}b}, 
\end{equation}
then $S = 0$.
\end{prop}
 
\begin{proof}
From the equation \eqref{basic eq: 3}, it follows that $S$ is a weighted shift with respect to the orthonormal basis $\nolinebreak\{z^{n} : n \in \mathbb{Z}\}$ in $L^{2}(\mathbb{T})$. Let $\{\alpha_{n}\}$ be the weight sequence of $S$. 
Putting $m = n-1$ in the equation \eqref{basic eq: 4}, then comparing the coefficient of $r,$ we have $2 \alpha_{n-1} s = 0$. Since $s \neq 0,$ it follows that $\alpha_{n-1} = 0$. This implies that $S=0.$
\end{proof}

\begin{cor}\label{all homog op for principal series case when s not 0}
If $T$ is a homogeneous operator with associated representation $P_{\lambda, s} \oplus P_{\lambda, s},$ where $s \neq 0,$ then, upto unitary equivalence, $T$ must be of the form    
$$\left[\begin{array}{ccc}
                    B(s) & \alpha (B(s) - B)\\
                      0 & B
                  \end{array}\right], \left[\begin{array}{ccc}
                    B(s) & 0\\
                      0 & B(s)
                  \end{array}\right]\,\,\mbox{\rm or}\,\,  \left[\begin{array}{ccc}
                      B & 0\\
                      0 & B
                  \end{array}\right].$$              
\end{cor}

\begin{proof}
Let $T$ be a homogeneous operator with associated representation $P_{\lambda, s} \oplus P_{\lambda, s}$. Recall that $P_{\lambda, s}$ and $P_{\lambda, -s}$ are unitarily equivalent via the unitary operator $U_{\lambda, s}.$  Clearly, the operator $\left(I \oplus U_{\lambda, s} \right) T \left(I \oplus U_{\lambda, s}^* \right)$ is homogeneous with associated representation $P_{\lambda, s} \oplus P_{\lambda, -s}$. Then Theorem \ref{homog op when associated rep is direct sum of cont. series rep} implies that $\left(I \oplus U_{\lambda, s} \right) T \left(I \oplus U_{\lambda, s}^* \right)$ is of the form $\left[\begin{array}{ccc}
                      \tilde{T_1} & \tilde{S_1}\\
                      0 & \tilde{T_2}
                  \end{array}\right]$ on $L^{2}(\mathbb{T}) \oplus L^{2}(\mathbb{T}),$ therefore the operator $T$ is also of the form $\left[\begin{array}{ccc}
                      T_1 & S_1\\
                      0 & T_2
                  \end{array}\right]$ on $L^{2}(\mathbb{T}) \oplus L^{2}(\mathbb{T})$. Now \cite[Proposition 2.4]{CONSTCharGMBB} and \cite[Lemma 2.5]{CONSTCharGMBB} imply that $T_1$ and $T_2$ are homogeneous operators with associated representation $P_{\lambda, s}$ and $S$ satisfies
\begin{equation*}
S\pi_{\lambda, s}(\phi) - e^{i \theta} \pi_{\lambda, s}(\phi) S = \overline{a}T_{1}\pi_{\lambda, s}(\phi)S + \overline{a} S \pi_{\lambda, s}(\phi) T_{2}.
\end{equation*} 
Since $B(s)$ and $B$ are the only homogeneous operators with associated representation $P_{\lambda, s}$, the proof is complete applying Proposition \ref{Finding S}, Proposition \ref{Finding S when s=0} and Proposition \ref{when s not 0, if diagonals are B(s) then S = 0}.
\end{proof}

Now we characterize all homogeneous operators whose associated representation is $P_{\lambda, 0} \oplus P_{\lambda, 0}$ with $\lambda \neq 1$.

Let $\sigma = P_{\lambda, 0} \oplus P_{\lambda, 0}.$  For all $i, j \in \mathbb{Z},$ let $\sigma_{i,j} = P_{i} \sigma|_{H(j)}$ where $P_{i}$ is the orthogonal projection of $L^2(\mathbb{T}) \oplus L^2(\mathbb{T})$ onto $H(i),$ the $\mathbb K$-isotypic subspace of $\sigma$ as in Theorem \ref{Irrducibility of operators corresponding to principal series rep}. Then $\sigma_{i,j}$ is a map from $H(j)$ to $H(i),$ $i,j\in \mathbb Z.$  Let $P_{\lambda, 0}^{i,j}$ be the map from the subspace of $L^2(\mathbb{T})$ spanned by the vector $\left\lbrace z^j \right\rbrace$ to the subspace of $L^2(\mathbb{T})$ spanned by the vector $\left\lbrace z^i \right\rbrace$ defined by $P_{\lambda, 0}^{i,j} (z^{j}) =  \left\langle P_{\lambda, 0}^{i,j} z^{j}, z^{i} \right\rangle z^{i}$. Then 
\begin{equation}\label{eqn:2.1.1.1.1}
\sigma _{i,j}(\phi) \binom{a z^{j}}{b z^{j}} = \left\langle P_{\lambda, 0}(\phi) z^{j}, z^{i}\right\rangle \binom{a z^{i}}{b z^{i}},
\end{equation}
for all $a, b \in \mathbb{C}$. Recall that the matrix coefficient of $P_{\lambda, 0}$ is
\begin{equation}
\left\langle P_{\lambda, 0}(\phi_{a})z^{m}, z^{n} \right\rangle = c (-1)^{n} (\overline{a})^{n - m} \sum_{k \geq (m - n)^{+}} C_{k}(m, n) r^{k},
\end{equation}
where $r = |a|^{2}$, $c = \phi_{a}^{'}(0)^{\lambda/2} |\phi_{a}^{'}(0)|^{\mu}$ and 
$C_{k}(m, n) = \left(\begin{array}{cc}
 -\lambda - \mu - m\\
 k + n -m
\end{array}\right)\left(\begin{array}{cc}
 - \mu + m\\
 k 
\end{array}\right).$ 
\begin{defn}\label{zero set of matrix coefficient}
Let $A_{m,n}$ be the subset of the interval $(-1, 1),$ which contains all zeros of the power series $\displaystyle \sum_{k \geq (m - n)^{+}} C_{k}(m, n) r^{k}.$ 
\end{defn}
Since for every $n, m \in \mathbb{Z}$, the radius of convergence of the power series $\displaystyle \sum_{k \geq (m - n)^{+}} C_{k}(m, n) r^{k}$ is $1,$ it follows $A_{m,n}$ is countable. Thus the set $A = \displaystyle \bigcup_{m,n \in \mathbb{Z}} A_{m,n}$ is also countable. Therefore, there exists $b \in (0, 1) \setminus A$ such that 
$\left\langle P_{\lambda, 0}(\phi_b)z^{m}, z^{n} \right\rangle \neq 0$, for all $n, m \in \mathbb{Z}.$ In the following, we fix this $\phi_b$ and  let $e_n$ denote the function $z^n$. 

Now assume that $u_{0}, v_{0}$ are two non-zero mutually orthogonal vectors in $H(0)$. Define $u_{n} = \sigma_{n,0}(\phi_b) u_{0}$, $v_{n} = \sigma_{n,0}(\phi_b) v_{0}$ for all $n \neq   0$. 
Then each of the vectors $u_{n}, v_{n}$ are non-zero.

\begin{lem}\label{u_n, v_n form an orthogonal basis}
The set of vectors $\left\lbrace u_{n}, v_{n}\right\rbrace_{n \in \mathbb{Z}}$ is a complete orthogonal set of $L^2(\mathbb{T}) \oplus L^2(\mathbb{T})$.
\end{lem}

\begin{proof}
As $u_{n}, v_{n} \in H(n)$ for every $n \in \mathbb{Z}$ and $H(n)$ is orthogonal to $H(m)$, so $\{u_{n}, v_{n}\}$ is orthogonal to $\{u_{m}, v_{m}\}$, if $ n \neq m$. Now we show that $u_{n}$ is  orthogonal to $v_{n},$ $n\in \mathbb Z.$  From the definition of  $\sigma_{n,0}(\phi_b): H(0) \to H(n)$ obtained from \eqref{eqn:2.1.1.1.1}
and a similar one for $\nolinebreak\sigma_{n,0}(\phi_b)^{*}: H(n) \to H(0),$ we have
$$\sigma_{n,0}(\phi_b)^{*} \sigma_{n,0}(\phi_b) = |\left\langle P_{\lambda, 0}(\phi_b)e_{0}, e_{n} \right\rangle|^{2} Id.$$ Consequently, 
\begin{flalign*}
\left\langle u_{n}, v_{n} \right\rangle &= \left\langle \sigma_{n,0}(\phi_b)u_{0}, \sigma_{n,0}(\phi_b)v_{0}\right\rangle \\
&= \left\langle \sigma_{n,0}(\phi_b)^{*} \sigma_{n,0}(\phi_b)u_{0}, v_{0} \right\rangle\\
&= |\left\langle P_{\lambda, 0}(\phi_b)e_{0}, e_{n} \right\rangle|^{2} \left\langle u_{0}, v_{0} \right\rangle = 0.
\end{flalign*}
Since $H(n)$ is spanned by $\left\lbrace u_{n}, v_{n}\right\rbrace$ and $L^2(\mathbb{T}) \oplus L^2(\mathbb{T}) = \displaystyle \oplus_{n \in \mathbb{Z}} H(n)$, it follows that $\left\lbrace u_{n}, v_{n}\right\rbrace_{n \in \mathbb{Z}}$ is a complete orthogonal set.
\end{proof}

Now let $H_{1}$ be the subspace of $L^2(\mathbb{T}) \oplus L^2(\mathbb{T})$ spanned by the set of vectors $\left\lbrace u_{n} \right\rbrace_{n \in \mathbb{Z}}$ and $H_2$ be the subspace of $L^2(\mathbb{T}) \oplus L^2(\mathbb{T})$ spanned by the set of vectors $\left\lbrace v_{n} \right\rbrace_{n \in \mathbb{Z}}$.

\begin{lem}\label{H_1, H_2 are invariant under sigma}
The subspaces $H_{1}$ and $H_2$ are invariant under $\sigma$. Moreover, $\sigma_{| H_{i}}$ is equivalent to $P_{\lambda, 0}$ for $i = 1, 2$.
\end{lem}

\begin{proof}
Pick a $\psi $ in  M\"{o}b. For all  $i, j$,  note that $\left\langle P_{\lambda, 0}(\phi_b)z^j, z^{i} \right\rangle \neq 0,$ 
and 
\begin{equation*}
\sigma_{i,j}(\psi) = \frac{\left\langle  P_{\lambda, 0}(\psi)e_{j}, e_{i} \right\rangle}{\left\langle  P_{\lambda, 0}(\phi_b)e_{j}, e_{i} \right\rangle} \sigma_{i,j}(\phi_b). 
\end{equation*}
Since 
\begin{equation*}
\sigma_{0,j}(\phi_b)\sigma_{j,0}(\phi_b) = \left\langle P_{\lambda, 0}(\phi_b)e_{0}, e_{j}\right\rangle \left\langle P_{\lambda, 0}(\phi_b)e_{j}, e_{0}\right\rangle \mathrm{Id}, 
\end{equation*}
it follows that $\sigma_{0,j}(\phi_b)u_{j}$ is in the span of $\{u_{0}\}$. Therefore,  
$$\sigma_{i,j}(\phi_b) u_{j} = \frac{\left\langle P_{\lambda, 0}(\phi_b)e_{j}, e_{i}\right\rangle}{\left\langle P_{\lambda, 0}(\phi_b)e_{j}, e_{0}\right\rangle \left\langle  P_{\lambda, 0}(\phi_b)e_{0}, e_{i}\right\rangle} \sigma_{i,0}(\phi_b) \sigma_{0,j}(\phi_b) u_{j}$$
is a scalar multiple of $u_{i}$. This implies that 
$$\sigma_{i,j}(\psi)u_{j} = \frac{\left\langle P_{\lambda, 0}(\psi)e_{j}, e_{i} \right\rangle}{\left\langle P_{\lambda, 0}(\phi_b)e_{j}, e_{i} \right\rangle} \sigma_{i,j}(\phi_b)u_{j}$$
is a scalar multiple of $u_{i}$. We conclude that $\sigma (\psi) u_{j}$ = $\sum _{i \in \mathbb{Z}} \sigma_{i,j}(\psi)u_{j}$ is in $H_{1},$ proving that $H_{1}$ is invariant under $\sigma$.  A similar argument shows that $H_{2}$ is invariant under $\sigma$.

Let $t_{n} \in \mathbb{R}$ be such that $\left\langle P_{\lambda, 0}(\phi_b)e_{0}, e_{n} \right\rangle$ = $e^{it_{n}}$ $|\left\langle P_{\lambda, 0}(\phi_b)e_{0}, e_{n} \right\rangle|$. Now if $\psi$ is any element in M\"{o}b, then 
\begin{multline*}
\left\langle \sigma(\psi)u_{j}, u_{i} \right\rangle = \left\langle \sigma_{i,j}(\psi)u_{j}, u_{i} \right\rangle = \left\langle P_{\lambda, 0}(\psi)e_{j}, e_{i}\right\rangle \left\langle P_{\lambda, 0}(\phi_b)e_{0}, e_{j} \right\rangle \overline{\left\langle P_{\lambda, 0}(\phi_b)e_{0}, e_{i} \right\rangle} \|u_{0}\|^{2}\\
= \left\langle P_{\lambda, 0}(\psi)e_{j}, e_{i} \right\rangle e^{it_{j}}\,\,|\left\langle P_{\lambda, 0}(\phi_b)e_{0}, e_{j}\right\rangle|\,\, e^{-it_{i}}\,\,|\left\langle P_{\lambda, 0}(\phi_b)e_{0}, e_{i} \right\rangle| \|u_{0}\|^{2}. 
\end{multline*}
Find $a, b \in \mathbb{C}$ such that $u_{0} = \binom{ae_{0}}{be_{0}}.$ Note that  
$$u_{n} = \sigma_{n,0}(\phi_b)u_{0} = \left\langle P_{\lambda, 0}(\phi_b)e_{0}, e_{n} \right\rangle \binom{ae_{n}}{be_{n}}\,\,\mbox{and, therefore,}\,\,  \|u_{n}\| = |\left\langle P_{\lambda, 0}(\phi_b)e_{0}, e_{n}\right\rangle| \|u_{0}\|.$$
The set of vectors $\{\hat{u}_{i}\},$  $\hat{u}_{i} = e^{-it_{i}}\frac{u_{i}}{\|u_{i}\|},$ is an orthonormal basis of $H_{1}$. From the preceding computation,  we see that  $\left\langle \sigma(\psi)\hat{u}_{j}, \hat{u}_{i} \right\rangle$ = $\left\langle P_{\lambda, 0}(\psi)e_{j}, e_{i} \right\rangle.$ It is now evident that $\sigma_{|H_{1}}$ is equivalent to  $P_{\lambda, 0}$. Similarly, it can be seen that $\sigma_{|H_2}$ is equivalent to $P_{\lambda, 0}$.
\end{proof}

Suppose $T$ is a homogeneous operator with associated representation $\sigma$. 
Since $H(n)$ is a $\mathbb{K}$-isotypic subspace of $\sigma$ and $\sigma$ is associated with $T$, therefore, we have $T (H(n)) \subseteq H(n+1)$ (\cite[Theorem 5.1]{Homogshift}). Let $T_{n}$ $:=$ $T_{|H(n)}$. We first prove that each $T_{n}$ is invertible.

\begin{lem}\label{invertibility of blocks}
For every $n \in \mathbb{Z}$, the operator $T_{n}$ is invertible.
\end{lem}

\begin{proof}
Let $\psi(z) = e^{i\theta}\frac{z - a}{1 - \overline{a}z}.$  The homogeneity of $T$ implies that
$$e^{i\theta}\sigma(\psi)T - ae^{i\theta}\sigma(\psi) = T\sigma(\psi) - \overline{a} T \sigma(\psi) T.$$ From this equation, using the orthogonality of the subspaces  $H(n),$ we have 
\begin{equation}\label{eqn:2.1.25}
e^{i\theta}\sigma_{i+1, n+1}(\psi) T_{n} - ae^{i\theta}\sigma_{i+1,n}(\psi) = T_{i} \sigma_{i,n}(\psi) - \overline{a}T_{i}\sigma_{i, n+1}(\psi)T_{n} 
\end{equation}
for all $i, n \in \mathbb{Z}$.

For all $i, j \in \mathbb{Z},$ the operator $\sigma_{i, j}(\phi_b)$ is invertible. Substituting $i = n$ and $\psi=\phi_b$ in the equation \eqref{eqn:2.1.25}, we get
\begin{equation*}
b\sigma_{n+1,n}(\phi_b) + \sigma_{n+1, n+1}(\phi_b) T_{n} = T_{n} \sigma_{n,n}(\phi_b) - \overline{b}T_{n}\sigma_{n, n+1}(\phi_b)T_{n}. 
\end{equation*}
If there exists $h_{n} \in H(n)$ such that $T_{n} h_{n}= 0$, then from the  equation appearing above, we have 
\begin{equation*}
b\sigma_{n+1, n}(\phi_b)h_{n} = 0 
\end{equation*}
and consequently, $h_{n} = 0$. This proves that $T_{n}$ is invertible.
\end{proof}

\begin{thm}\label{exsistence of simultaneous invariant subspace of T and sigma}
Suppose $T$ is a homogeneous operator with associated representation $\sigma$. Then there exists $H_{1}$ and $H_2$ such that $L^2(\mathbb{T}) \oplus L^2(\mathbb{T}) = H_{1} \oplus H_{2}$, $T (H_{1}) \subseteq H_{1}.$  The subspaces $H_{i},$ $i = 1, 2,$ are invariant under $\sigma$ and $\sigma_{|H_{i}}$ is unitarily equivalent to $P_{\lambda, 0}.$ 
\end{thm}

\begin{proof}There exists $\lambda_{0} \in \mathbb{C}$ and a pair of orthonormal vectors $u_0, v_0$ in $H(0)$ such that the vector $u_0$ is an eigenvector for the operator 
$\sigma_{1,0}(\phi_b)^{-1} T_{0}$ with eigenvalue $\lambda_0,$ that is,  
$$\sigma_{1,0}(\phi_b)^{-1} T_{0} u_{0} = \lambda_{0} u_{0}.$$ 
Now, define 
$$u_{n} = \sigma_{n,0}(\phi_b)u_{0},\,\, v_{n} = \sigma_{n,0}(\phi_b)v_{0}$$
for all $n \in \mathbb{Z},\,\,n \neq 0$. Suppose $H_{1}$ and $H_{2}$ are the closed subspaces spanned by $\{u_{n}\}_{n \in \mathbb{Z}}$ and $\{v_{n}\}_{n \in \mathbb{Z}},$ respectively. Then by Lemma \ref{u_n, v_n form an orthogonal basis},  $L^2(\mathbb{T}) \oplus L^2(\mathbb{T}) = H_{1} \oplus H_{2}$ and by Lemma \ref{H_1, H_2 are invariant under sigma}, each $H_{i}$ is invariant under $\sigma$ such that $\sigma_{|H_{i}}$ is equivalent to $P_{\lambda, 0}$. Now we show that $T (H_{1}) \subseteq H_{1}$.

We have $T_{0}u_{0} = \lambda_{0}\sigma_{1,0}(\phi_b)u_{0},$ which is a scalar multiple of the vector $u_{1}$. An inductive argument given below shows that $T_{n} u_{n}$ is a scalar multiple of the vector $u_{n+1}$ for every $n \in \mathbb{Z}$.

Assume that $T_{k}u_{k} = \lambda_{k+1}u_{k+1}$ for some $\lambda_{k+1} \in \mathbb{C},$ $k \geq 0$. Let $A_{k} = \displaystyle\bigcup_{0 \leq i,j \leq k+2}A_{i,j},$ where $A_{i,j}$ are described in Definition \ref{zero set of matrix coefficient}. Since $0$ is not a limit point of any $A_{i, j}$, 
  there exists $r_{k} \in (0, 1)$ such that $\left\langle P_{\lambda, 0}(\phi_{a})z^{j}, z^{i} \right\rangle \neq 0$,  $0 \leq i, j \leq k+2,$
for all $a \in \mathbb{D}$ with $0 < |a| < r_{k}.$ 
Combining the two equalities 
$$\sigma_{k+1, 0}(\phi_b) = \frac{\left\langle P_{\lambda, 0}(\phi_b)e_{0}, e_{k+1}\right\rangle}{\left\langle P_{\lambda, 0}(\phi_b)e_{0}, e_{k}\right\rangle \left\langle P_{\lambda, 0}(\phi_b)e_{k}, e_{k+1} \right\rangle}\sigma_{k+1,k}(\phi_b)\sigma_{k,0}(\phi_b)$$
and 
$$\sigma_{k+1,k}(\phi_b) = \frac{\left\langle P_{\lambda, 0}(\phi_b)e_{k}, e_{k+1}\right\rangle }{\left\langle P_{\lambda, 0}(\phi_{a})e_{k}, e_{k+1} \right\rangle}\sigma_{k+1,k}(\phi_{a}),\,\,|a| < r_{k}, $$
we have $T_{k}u_{k} = \lambda_{k+1}(a)\sigma_{k+1,k}(\phi_{a})u_{k}$, where 
$$\lambda _{k+1}(a) = \lambda_{k+1}   \frac{\left\langle P_{\lambda, 0}(\phi_b)e_{0}, e_{k+1}\right\rangle \left\langle P_{\lambda, 0}(\phi_b)e_{k}, e_{k+1}\right\rangle}{\left\langle P_{\lambda, 0}(\phi_b)e_{0}, e_{k}\right\rangle \left\langle P_{\lambda, 0}(\phi_b)e_{k}, e_{k+1}\right\rangle \left\langle P_{\lambda, 0}(\phi_{a})e_{k}, e_{k+1} \right\rangle}.$$
For every $\phi_{a}$ with $|a| < r_{k}$, this proves the existence of $\lambda_{k+1}(a) \in \mathbb{C}$ such that 
$$T_{k}u_{k} = \lambda_{k+1}(a) \sigma_{k+1, k}(\phi_{a})u_{k}.$$
Now, for every $\phi_{a}$ with $|a| < r_{k}$, substituting $n = k, i = k+1$ in the equation \eqref{eqn:2.1.25}, and then evaluating on the vector $u_k,$ we get
\begin{multline*}
a\sigma_{k+2,k}(\phi_{a})u_{k} - \lambda_{k+1}(a)\sigma_{k+2, k+1}(\phi_{a}) \sigma_{k+1, k}(\phi_{a})u_{k} \\
= T_{k+1} \sigma_{k+1,k}(\phi_{a})u_{k} - \overline{a}\lambda_{k+1}(a)\left\langle P_{\lambda, 0}(\phi_{a})e_{k+1}, e_{k+1} \right\rangle T_{k+1} \sigma_{k+1,k}(\phi_{a})u_{k}.
\end{multline*}
The equality below is easily verified using the definition of the $\sigma_{i,j}:$ 
\begin{equation*}
\sigma_{k+2, k}(\phi_{a}) = \frac{\left\langle P_{\lambda, 0}(\phi_{a})e_{k}, e_{k+2} \right\rangle}{\left\langle P_{\lambda, 0}(\phi_{a})e_{k}, e_{k+1}\right\rangle \left\langle P_{\lambda, 0}(\phi_{a})e_{k+1}, e_{k+2} \right\rangle} \sigma_{k+2, k+1}(\phi_{a})\sigma_{k+1, k}(\phi_{a}).
\end{equation*}
In consequence, we have 
\begin{multline}\label{equation:2.2.2.x}
\left(\frac{a \left\langle P_{\lambda, 0}(\phi_{a})e_{k}, e_{k+2} \right\rangle}{\left\langle P_{\lambda, 0}(\phi_{a})e_{k}, e_{k+1}\right\rangle \left\langle P_{\lambda, 0}(\phi_{a})e_{k+1}, e_{k+2}\right\rangle} - \lambda_{k+1}(a)\right) \sigma_{k+2, k+1}(\phi_{a})\sigma_{k+1, k}(\phi_{a})u_k\\
= \left(1 - \overline{a} \lambda_{k+1}(a)\left\langle P_{\lambda, 0}(\phi_{a})e_{k+1}, e_{k+1} \right\rangle\right) 
T_{k+1}\sigma_{k+1,k}(\phi_{a}) u_{k}.
\end{multline}
Suppose 
\begin{equation*}
\left(\frac{a \left\langle P_{\lambda, 0}(\phi_{a})e_{k}, e_{k+2}\right\rangle }{\left\langle P_{\lambda, 0}(\phi_{a})e_{k}, e_{k+1}\right\rangle \left\langle P_{\lambda, 0}(\phi_{a})e_{k+1}, e_{k+2}\right\rangle} - \lambda_{k+1}(a)\right) = 0
\end{equation*}
and
\begin{equation*}
\left(1 - \overline{a} \lambda_{k+1}(a)\left\langle P_{\lambda, 0}(\phi_{a})e_{k+1}, e_{k+1}\right\rangle\right) = 0
\end{equation*}
for all $\phi_a$ with $|a| < r_k$.
Then we have
\begin{equation*}
|a|^{2} \left\langle P_{\lambda, 0}(\phi_{a})e_{k}, e_{k+2}\right\rangle \left\langle P_{\lambda, 0}(\phi_{a})e_{k+1}, e_{k+1}\right\rangle = \left\langle P_{\lambda, 0}(\phi_{a})e_{k}, e_{k+1}\right\rangle \left\langle P_{\lambda, 0}(\phi_{a})e_{k+1}, e_{k+2}\right\rangle
\end{equation*}
for all $|a| < r_{k}$. 


Now, using the matrix coefficient for $P_{\lambda, 0}(\phi_{a})$,  $0 \leq r \leq r_{k}^{2},$ and then putting $r = 0$ we arrive at a contradiction.

We can therefore find $\phi_{a}$ with $0 < |a| < r_{k}$ such that 
$$\left(\frac{a \left\langle P_{\lambda, 0}(\phi_{a})e_{k}, e_{k+2}\right\rangle}{\left\langle P_{\lambda, 0}(\phi_{a})e_{k}, e_{k+1}\right\rangle \left\langle P_{\lambda, 0}(\phi_{a})e_{k+1}, e_{k+2}\right\rangle} - \lambda_{k+1}(a)\right) \neq 0$$ 
and hence 
$$\left(1 - \overline{a} \lambda_{k+1}(a)\left\langle P_{\lambda, 0}(\phi_{a})e_{k+1}, e_{k+1}\right\rangle\right) \neq 0$$
as both $\sigma_{k+2, k+1}(\phi_{a})$ and $T_{k+1}$ are invertible. 
Since $0 < |a| < r_{k}$, 
it follows from \eqref{equation:2.2.2.x} that $T_{k+1}u_{k+1}$ is a scalar multiple of the vector $u_{k+2}$
completing half the induction argument. 

A similar but slightly different proof gives the other half of the induction argument, namely, $T_{-n}^{-1}u_{-n+1}$ is a scalar multiple of $\{u_{-n}\}$ for all $n \in \mathbb{N}$.
\end{proof}

\begin{cor}\label{all homog op for principal series case when s = 0}
If $T$ is a homogeneous operator with associated representation $P_{\lambda, 0} \oplus P_{\lambda, 0},\,\,\lambda \neq 1,$ then $T$ is unitarily equivalent to one of the following     
operator
$$\left[\begin{array}{ccc}
                    B & \alpha S[\lambda]\\
                      0 & B
                  \end{array}\right], \left[\begin{array}{ccc}
                      B & 0\\
                      0 & B
                  \end{array}\right],$$  
where $S[\lambda]$ is the weighted shift on $L^2(\mathbb{T})$ with respect to the orthonormal basis $\left\lbrace z^n : n \in \mathbb{Z} \right\rbrace$ with weight sequence $\left\lbrace \frac{1}{\lambda + 2n + 1} : n \in \mathbb{Z} \right\rbrace$.                             
\end{cor}

\begin{proof}
The proof follows form Theorem \ref{exsistence of simultaneous invariant subspace of T and sigma} and Proposition \ref{Finding S when s=0}(b).
\end{proof}

\section{The associated representation is the direct sum of three irreducible representations}
Now, we prove that every homogeneous operator whose associated representation is $\pi = \pi_{1} \oplus \pi_{2},$ where $\pi_{1}$ is from the irreducible Continuous series representations and $\pi_{2}$ is the direct sum of a holomorphic and an anti-holomorphic Discrete series representation, is reducible. Let $\pi_{1} = R_{\lambda, \mu}$ and $H_{1}$ be the representation space of $\pi_{1}$. Let $e_{n}^{1} = \frac{z^{n}}{\|z^{n}\|_{1}}.$ Recall that $\{e_n^1:n\in \mathbb Z\}$ is an orthonormal basis of the representation space $H_1.$ 
Let $\pi_{2} = D_{\lambda_1}^{+} \oplus D_{\lambda_2}^{-}$ 
for a pair of positive real numbers $\lambda_1, \lambda_2.$ However, the  multipliers of all the three representations  $\pi_1,$ $D^+_{\lambda_1}$ and $D^-_{\lambda_2}$ must be the same. In consequence, $\lambda_1+\lambda_2$ is an even integer (see \cite[Corollary 3.2]{Homogshift}), therefore $\lambda_1= \lambda+2m$ and $\lambda_2 = 2-\lambda+2k,$ $- 1 <\lambda \leq 1.$  

Let $H^{(\lambda + 2m)}$ be the representation space of $D_{\lambda + 2m}^{+}$ and $H^{(2-\lambda + 2k)}$ be the representation space of $D_{2-\lambda + 2k}^{-}$. Let $H_{2} = H^{(\lambda + 2m)} \oplus H^{(2-\lambda + 2k)}$. 
The set of vectors $\{e_{n}^{2} : n \in \mathbb{Z} \}$ is an orthonormal basis of $H_{2}$, where $e_n^2,\,\,n \in \mathbb{Z},$ are described in \eqref{eq: onb}. Let $\phi_{\theta}$ be a rotation in M\"{o}b. Then 
$$\pi_{1}(\phi_{\theta}) e_{n}^{1} = e^{-i \left( n + \frac{\lambda}{2} \right) \theta} e_{n}^{1},\, n \in \mathbb{Z}.$$
Also, it is easy to see that 
$$\pi_{2}(\phi_{\theta}) e_{n}^{2} = e^{-i \left( n + m + \frac{\lambda}{2} \right) \theta}e_{n}^{2},\,n \geq 0\,\,\mbox{and}\,\,\pi_{2}(\phi_{\theta}) e_{-n}^{2} = e^{i \left( n + k - \frac{\lambda}{2} \right) \theta}e_{-n}^{2},\,n \geq 1.$$
Clearly, there exists a $\theta$ such that $e^{-i \left( n + m + \frac{\lambda}{2} \right) \theta} \neq e^{i \left(p + k - \frac{\lambda}{2} \right) \theta}$ for all $n \geq 0$, $p \geq 1$ and if $n_{1} \neq n_{2},$ then $e^{-i \left( n_{1} + \frac{\lambda}{2} \right) \theta} \neq e^{-i \left( n_{2} + \frac{\lambda}{2} \right) \theta }.$

%
%

\begin{thm}\label{when rep. has three comp, S_1 = 0}
Suppose $T = \left[\begin{array}{ccc}
    T_{1} & S_{1}\\
    S_{2} & T_{2}
\end{array}\right]$ is a homogeneous operator with associated representation $\pi_{1} \oplus \pi_{2},$ where $\pi_1 = R_{\lambda, \mu}$ is from the Continuous series excluding $P_{1,0}$ and $\pi_{2} = D_{\lambda + 2m}^{+} \oplus D_{2-\lambda + 2k}^{-}$. Then $S_1 = 0$.
\end{thm}

\begin{proof}
Homogeneity of $T$ implies that the operators $T_{i}$ and $S_{i}$ satisfy equations \eqref{eqn:2.1.17}, \eqref{eqn:2.1.18}, \eqref{eqn:2.1.19} and \eqref{eqn:2.1.20}.
Repeating an argument similar to the one in Theorem \ref{homog op when associated rep is direct sum of cont. series rep}, we find that $T_{1}$ and $T_2$ are weighted shifts with respect to the orthonormal basis $\{e_{n}^{1} \}$ and $\{e_n^2\},$ respectively. Let $\{u_{n}\}$ and $\{v_n\}$ be the weight sequences of $T_{1}$ and $T_2,$ respectively. It is easy to see that $v_{-1} = 0$ unless either $m = 0,$ $k = 0,$ $\lambda >0$ or $m = 1,$ $k = -1,$ $\lambda < 0$ .

From the equation \eqref{basic eq: 5}, it follows that
\begin{enumerate}
\item[(a)] for $n \geq 0,$ there exists $\alpha_{n} \in \mathbb{C}$ such that $S_{1}e_{n}^{2} = \alpha_{n}e_{n+m+1}^{1}$. 
\item[(b)] for $n \geq 1,$ there exists $\alpha_{-n} \in \mathbb{C}$ such that $S_{1}e_{-n}^{2} = \alpha_{-n}e_{-n-k+1}^{1}.$

\end{enumerate}
Applying \hyperlink{algorithm 1}{Algorithm $1$} from the Appendix, for $i = n \geq 0$ and $j = -p-k+1,$  $p > 1$, using the matrix coefficient of $\pi_1(\phi_a)$ and finally comparing the coefficient of $r^{n+m+p+k},$ we obtain 

\begin{equation*}
\alpha_{n} \|z^{-p-k+1}\|_{1} C_{n+m+p+k}^{1}(n+m+1, -p-k+1) = 0 
\end{equation*}
This implies that $\alpha_{n} = 0$, because $\|z^{-p-k+1}\|_{1} C_{n+m+p+k-1}^{1}(n+m+1, -p-k+1) \neq 0$. This proves that $S_{1} e_{n}^{2} = 0$ for all $n \geq 0$.

To prove that $S_1 e_{-n}^2 = 0,$  $n \geq 1,$ we again apply \hyperlink{algorithm 1}{Algorithm $1$} for $i=n \leq -1$ and $j=p > 0,$ use the matrix coefficients of $\pi_1(\phi_a)$ and finally equate the constant term on both sides to conclude
\begin{equation*}
\alpha_{-n} \left[u_{p-1} \frac{\|z^{p-1}\|_{1}}{\|z^{p}\|_{1}} + \frac{(-\lambda - \mu - p + 1)}{(p + n + k - 1)} \right] = 0.
\end{equation*}
Now, suppose there exists a subsequence $(n_{m})$ such that $\alpha_{-n_{m}} \neq 0.$  Then 
\begin{equation*}
u_{p-1} \frac{\|z^{p-1}\|_{1}}{\|z^{p}\|_{1}} + \frac{(-\lambda - \mu - p + 1)}{(p + n_{m} + k - 1)}  = 0,
\end{equation*}
for all $n_{m}.$ Therefore taking $m \rightarrow \infty$, we see that 
$u_{p-1} \tfrac{\|z^{p-1}\|_{1}}{\|z^{p}\|_{1}} = 0.$
Hence $\alpha_{-n} = 0$ for all $n \geq 1,$ leading to a contradiction, since we have assumed that $\alpha_{-n_{m}} \neq 0$ for all $m \geq 1$. 
Thus there is no subsequence $\{n_{m}\}$ such that $\alpha_{-n_{m}} \neq 0,$ or in other words, there exists a natural number $N$ such that $\alpha_{-n} = 0$ for all $n \geq N$. One more time applying  \hyperlink{algorithm 1}{Algorithm $1$} for $-N < i=n \leq -1$ and $j = -n-l-k+2$ where  $l:l > N-n,$ then using the matrix coefficients of $\pi_1(\phi_a)$ and finally comparing coefficients of $r^l,$ we have 
\begin{equation*}
\alpha_{-n} \|z^{-n-l-k+1}\|_{1} C_{l}^{1}(-n-k+1, -n-l-k+1) = 0.
\end{equation*}
It follows that $\alpha_{-n} = 0$ for all $1 \leq n < N$. Therefore we have proved that $S_{1} = 0$.
\end{proof}

\begin{thm}\label{when rep. has three comp, and mu is not 0, S_2 = 0}
Suppose $T = \left[\begin{array}{ccc}
    T_{1} & 0\\
    S_{2} & T_{2}
\end{array}\right]$ is a homogeneous operator and $\pi_{1} \oplus \pi_{2}$ is the associated representation. Then $S_2$ satisfies the equation \eqref{eqn:2.1.19}.  If $\pi_1 = R_{\lambda, \mu}$ is from the Continuous series representation excluding $P_{1,0}$ and $\pi_{2} = D_{\lambda + 2m}^{+} \oplus D_{2-\lambda + 2k}^{-},$   then $S_2 = 0$.
\end{thm}

\begin{proof}
\textbf{Case I $\boldsymbol{(m \geq 1)}$:} Assume $m\geq 1.$  From 
the equation \eqref{basic eq: 6}, only the following possibilities occur.
\begin{enumerate}
\item[(a)] There exists $\alpha_{n} \in \mathbb{C}$ such that $S_{2}e_{n}^{1} = \alpha_{n}e_{n+1-m}^{2},$  $n \geq m - 1$ and $S_{2}e_{n}^{1} = 0,$  $0 \leq n < m-1.$
\item[(b)] There exists $\alpha_{-n} \in \mathbb{C}$ such that $S_{2}e_{-n}^{1} = \alpha_{-n}e_{-n+k+1}^{2},$ $n > k + 1$ and $S_{2}e_{-n}^{1} = 0,$  $1 \leq n \leq k + 1.$ 
\end{enumerate}
Applying \hyperlink{algorithm 2}{Algorithm $2$} from the appendix, for $i = 0$ and $j = -n+k+1,$ $n > k+1$, using the matrix coefficient of $\pi_1(\phi_a)$ and finally comparing the coefficient of $r^{n},$ we see that $\alpha_{-n} = 0$. Thus for every $n \geq 1,$ we have $S_{2}e_{-n}^{1} = 0.$ To complete the proof, we have to show that $S_{2}e_{n}^{1} = 0$, $n \geq 0$.

Now applying \hyperlink{algorithm 2}{Algorithm $2$} for $i = -1$ and $j = n+1-m,$  $n \geq m-1$, using the matrix coefficient of $\pi_1(\phi_a)$ and finally comparing the constant coefficients and the coefficients of $r,$ respectively, we get
\begin{equation*}
\frac{\alpha_{n} (- \lambda - \mu + 1)}{\|z^{-1}\|_{1} (n+1)} = \frac{\alpha_{n} u_{-1}}{\|z^{0}\|_{1}}\,\,\mbox{and}\,\,
\frac{\alpha_{n} (\mu + 1) (- \lambda - \mu + 1)}{\|z^{-1}\|_{1} (n+2)} = \frac{\alpha_{n} \mu u_{-1}}{\|z^{0}\|_{1}}.
\end{equation*}
These two equations together give  
\begin{equation*}
\alpha_{n} \left[\frac{(\mu + 1)}{(n+2)} - \frac{ \mu}{(n+1)}\right] = 0.
\end{equation*}
Since $\frac{(\mu + 1)}{(n + 2)} \neq \frac{\mu}{(n + 1)}$ for all $n \geq 0,$  we must have $\alpha_{n} = 0$ for all $n \geq 0$. This proves that $S_{2}e_{n}^{1} = 0$, $n \geq 0$. 

\noindent\textbf{Case II $\boldsymbol{(m = 0)}$:} Assume $m=0.$  In this case, $\lambda > 0.$ From the equation \eqref{basic eq: 6}, we see that 
\begin{enumerate}
\item[(a)] there exists $\alpha_n \in \mathbb{C}$ such that $S_{2}e_{n}^{1} = \alpha_{n}e_{n+1}^{2}$, $n \geq -1$;
\item[(b)] there exists $\alpha_{-n} \in \mathbb{C}$ such that $S_{2}e_{-n}^{1} = \alpha_{-n}e_{-n+k+1}^{2},$  $n > k + 1$ and $S_{2}e_{-n}^{1} = 0,$  $2 \leq n \leq k + 1.$ 
\end{enumerate}
Repeating a similar computation as in the case of ($m \geq 1$), we conclude that $\alpha_{n} = 0$ for all $n$. Therefore we have proved that $S_2 = 0$ in this case. 
\end{proof}

\begin{rem}\label{Remark about redicibility of T when the rep is P_1,s + P_1,0}
Suppose $T = \begin{bmatrix}
T_1 & S_1\\
S_2 & T_2
\end{bmatrix}$ is a homogeneous operator with associated representation $\pi_1 \oplus \pi_2,$ where $\pi_1$ is an irreducible Continuous series representation and $\pi_2 = P_{1, 0}$. The representation $P_{1, 0}$ is a direct sum of two irreducible representations. One of the summands is the $D_1^+$ while the other one is equivalent to $D_1^-.$ However, $P_{1, 0}$ is not equivalent to $D_1^+ \oplus D_1^-,$ therefore this case is not covered by the previous analysis. Fortunately, repeating the computations of Theorem \ref{when rep. has three comp, S_1 = 0} and Theorem \ref{when rep. has three comp, and mu is not 0, S_2 = 0}, we obtain that $S_1 = S_2 = 0.$
\end{rem}


\section{The associated representation is the direct sum of four irreducible representations}
In this section, we prove that every homogeneous operator with associated representation $\pi_1 \oplus \pi_2,$ where $\pi_1 = D_{\lambda_1}^+ \oplus D_{\lambda_2}^-$ and $\pi_2 = D_{\lambda_3}^+ \oplus D_{\lambda_4}^-$, is reducible. If $D_{\lambda_1}^+ \oplus D_{\lambda_2}^- \oplus D_{\lambda_3}^+ \oplus D_{\lambda_4}^-$ is a representation, then the  multipliers of all the four representations  $D_{\lambda_1}^+,$ $D_{\lambda_2}^-,$ $D_{\lambda_3}^+$ and $D_{\lambda_4}^-$ must be the same. In consequence, $\lambda_1= \lambda+2a,$ $\lambda_2 = 2-\lambda+2b,$ $\lambda_3 = \lambda+2m$ and $\lambda_4 = 2-\lambda+2p$ for some real $\lambda$ with  $0 < \lambda \leq 2$ and some non negative integers $a, b, m, p.$ 
  
Let $\lambda \in (0, 2]$ and $a, b, m, p$ be any non-negative integers. Let $\pi_{1} = D_{\lambda + 2a}^{+} \oplus D_{2-\lambda + 2b}^{-}$ and $\pi_{2} = D_{\lambda + 2m}^{+} \oplus D_{2-\lambda + 2p}^{-}$. Then the representation space of $\pi_{1}$ is $H_1:= H^{(\lambda + 2a)} \oplus H^{(2-\lambda + 2b)}$ and the representation space of $\pi_{2}$ is $H_2:= H^{(\lambda + 2m)} \oplus H^{(2-\lambda + 2p)}$.  
The vectors  $e_{n}^{i}, n \in \mathbb{Z}$ form an orthonormal  basis  of $H_{i}$, $i=1,2,$ where $e_n^i,\,\,n \in \mathbb{Z}$ is defined in a similar fashion as in \eqref{eq: onb}.  If $\phi_{\theta}$ is a rotation in M\"{o}b, then 
$$\pi_{1}(\phi_{\theta})e_{n}^{1} = e^{-i \left( n + a + \frac{\lambda}{2} \right) \theta} e_{n}^{1}, n \geq 0;\,\, \pi_{1}(\phi_{\theta})e_{-n}^{1} = e^{i \left( n + b - \frac{\lambda}{2} \right) \theta} e_{-n}^{1}, n \geq 1$$ 
and
$$\pi_{2}(\phi_{\theta})e_{n}^{2} = e^{-i \left( n + m + \frac{\lambda}{2} \right) \theta} e_{n}^{2}, n \geq 0;\,\,\pi_{2}(\phi_{\theta})e_{-n}^{2} = e^{i \left( n + p - \frac{\lambda}{2} \right) \theta} e_{-n}^{2}, n \geq 1.$$
We can, therefore, find $\theta$ such that $\pi_{1}(\phi_{\theta})$ and $\pi_{2}(\phi_{\theta})$ have distinct eigenvalues with one dimensional eigenspaces described as above.

\begin{lem}\label{lamda is not 1}
Let $T = \left[\begin{array}{ccc}
    T_{1} & S_{1}\\
    S_{2} & T_{2}
\end{array}\right]$ be a homogeneous operator with associated representation $\pi_{1} \oplus \pi_{2},$ where $\pi_{1} = D_{\lambda + 2a}^{+} \oplus D_{2-\lambda + 2b}^{-}$ and $\pi_{2} = D_{\lambda + 2m}^{+} \oplus D_{2-\lambda + 2p}^{-}$. Then the following holds:
\begin{enumerate}
\item[(a)] For $n \in \mathbb{Z},$ there exist $u_{n} \in \mathbb{C}$ such that $T_1e_{n}^1 = u_{n} e_{n+1}^{1}$, where $u_{-1} = 0$ unless $a = 0$ and $b = 0$.
\item[(b)] For $n \in \mathbb{Z},$ there exists $v_{n} \in \mathbb{C}$ such that $T_{2}e_{n}^{2} = v_{n}e_{n+1}^{2}$, where $v_{-1} = 0$ unless $m = 0$ and $p = 0$.
\item[(c)] For $n \geq 0,$ $S_{1}e_{n}^{2}$ belongs to the span closure of the set of vectors $\{ e_{q}^{1} : q \geq 0 \}$ and for $\nolinebreak{n \geq 2},$  $S_{1}e_{-n}^{2}$ belongs to the span closure of the set of vectors $\{e_{-q}^{1} : q \geq 1 \}.$ The vector $S_{1}e_{-1}^{2}$ belongs to the span closure of the set $\{ e_{-q}^{1} : q \geq 1 \}$ unless $p = 0$ and $a = 0$.
\item[(d)] For $n \geq 0,$ $S_{2}e_{n}^{1}$ belongs to the span closure of the set of vectors $\{ e_{q}^{2} : q \geq 0 \}$ and for $n \geq 2,$ $S_{2}e_{-n}^{1}$ belongs to the span closure of the set $\{e_{-q}^{2} : q \geq 1 \}$. The vector $S_{2}e_{-1}^{1}$ belongs to the span closure of the set $\{ e_{-q}^{2} : q \geq 1 \}$ unless $b = 0$ and $m = 0$.
\end{enumerate}
\end{lem}

\begin{proof}
Homogeneity of $T$ implies that $T_{i}$ and $S_{i}$ satisfy equations \eqref{eqn:2.1.17}, \eqref{eqn:2.1.18}, \eqref{eqn:2.1.19} and \eqref{eqn:2.1.20}.
Substituting $\phi=\phi_{\theta}$ in the equation \eqref{eqn:2.1.17}, we get 
$$\pi_{1}(\phi_{\theta}) T_{1} e_{n}^{1} = e^{- i \left( n + 1 + a + \frac{\lambda}{2} \right) \theta} T_{1}e_{n}^{1}, n \geq 0\,\,\mbox{and}\,\,\pi_{1}(\phi_{\theta}) T_{1} e_{-n}^{1} = e^{i \left( n - 1 + b - \frac{\lambda}{2} \right) \theta} T_{1}e_{-n}^{1}, n \geq 1.$$ 
Therefore, for each $n \in \mathbb{Z},$ there exists $u_{n} \in \mathbb{C}$ such that $T_1e_{n}^1 = u_{n} e_{n+1}^{1},$  $u_{-1} = 0,$ unless $a = 0$ and $b = 0$.


Similarly, we can show that for all $n \in \mathbb{Z},$ there exists $v_{n} \in \mathbb{C}$ such that $T_{2}e_{n}^{2} = v_{n}e_{n+1}^{2},$  $v_{-1} = 0,$ unless $m = 0$ and $p = 0$.
Now from the equation \eqref{basic eq: 7}, we obtain 
\begin{enumerate}
\item for each $n\geq 0,$ $S_{1}e_{n}^{2}$ belongs to the span closure of the set of vectors $\{ e_{q}^{1} : q \geq 0 \},$    
\item for each $n \geq 2,$ $S_{1}e_{-n}^{2}$ belongs to the span closure of the set of vectors $\{e_{-q}^{1} : q \geq 1 \}$ and 
\item except when  $p = 0$ and $a = 0,$ $S_{1}e_{-1}^{2}$ belongs to the span closure of the set of vectors $\{ e_{-q}^{1} : q \geq 1 \}.$
\end{enumerate}

The proof of part (d) is similar to the proof of part (c).
\end{proof}

\begin{lem}\label{b and m are 0}
Suppose $T = \left[\begin{array}{ccc}
    T_{1} & S_{1}\\
    S_{2} & T_{2}
\end{array}\right]$ is a homogeneous operator with associated representation $\pi_{1} \oplus \pi_{2},$ where $\pi_{1} = D_{\lambda+2a}^{+} \oplus D_{2-\lambda}^{-}$ and $\pi_{2} = D_{\lambda}^{+} \oplus D_{2-\lambda+2p}^{-},$ for a pair $a, p$ of positive integers. Then $T$ is reducible. Furthermore, $T = \tilde{T_{1}} \oplus \tilde{T_{2}}$ where $\tilde{T_{1}}$ is a homogeneous operator with associated representation $D_{\lambda + 2a}^{+} \oplus D_{\lambda}^{+}$ and $\tilde{T_{2}}$ is a homogeneous operator with associated representation $D_{2-\lambda}^{-} \oplus D_{2-\lambda + 2p}^{-}$.
\end{lem}

\begin{proof} 
Homogeneity of $T$ implies that the operators $T_{i}$ and $S_{i}$ satisfy equations \eqref{eqn:2.1.17}, \eqref{eqn:2.1.18}, \eqref{eqn:2.1.19} and \eqref{eqn:2.1.20}.
Since $a \neq 0$ and $p \neq 0$, from Lemma \ref{lamda is not 1}, it follows  that
\begin{enumerate}
\item[(a)] for $n\geq 0$, $T_{i}e_{n}^{i}$ is in the span closure of $\{e_{q}^{i} : q \geq 0 \}$,  $i=1,2,$
\item[(b)] for $n\geq 1$, $T_{i}e_{-n}^{i}$ is in the span closure of $\{e_{-q}^{i} : q \geq 1 \}$, $i=1,2,$
\item[(c)] for $n \geq 0,$ $S_{1}e_{n}^{2}$ is in the span closure of $\{e_{q}^{1} : q \geq 0 \}$ and 
\item[(d)]  for $n \geq 1$, $S_{1}e_{-n}^{2}$ is in the span closure of $\{e_{-q}^{1} : q \geq 1 \}$.
\end{enumerate}
From the equation \eqref{basic eq: 8}, it follows that (i) for  $n\geq 0,$ there exists $\alpha_n \in \mathbb{C}$ such that $S_{2}e_{n}^{1} = \alpha_{n} e_{n+1+a}^{2},$ (ii) for $n\geq p+2,$ there exists $\alpha_{-n} \in \mathbb{C}$ such that $S_{2} e_{-n}^{1} = \alpha_{-n}e_{-n+p+1}^{1},$ (iii) for $2 \leq n \leq p + 1,$ $S_{2} e_{-n}^{1} = 0$ and (iv) there exists $\alpha_{-1} \in \mathbb{C}$ such that $S_{2}e_{-1}^{1} = \alpha_{-1}e_{0}^{2}$.

%
Now applying \hyperlink{algorithm 2}{Algorithm $2$} from the Appendix, for $i = -1$ and $j = 0$, we obtain
\begin{equation*}
\alpha_{-1} \left\langle D_{2-\lambda}^{+}(\phi_{a}^{*}) z^{0}, z^{0} \right\rangle + \alpha_{-1} \left\langle D_{\lambda}^{+}(\phi_{a}) z^{0}, z^{0} \right\rangle = 0. 
\end{equation*}
If $a$ is real, then $\phi_{a}^{*} = \phi_{a}$. An easy computation shows that  $\left\langle D_{2-\lambda}^{+}(\phi_{a}) z^{0}, z^{0} \right\rangle + \left\langle D_{\lambda}^{+}(\phi_{a}) z^{0}, z^{0} \right\rangle \neq 0$, $a \in (0,1).$ In  consequence $\alpha_{-1} = 0$.

Let $\tilde{H_{1}}$ and $\tilde{H_{2}}$ be the closed subspaces of $H$  spanned by the orthonormal set of vectors 
\begin{equation} \label{eqn:2.31a}
\left\lbrace \left(\begin{array}{ccc}
    e_{n}^{1}\\
     0
\end{array}\right), \left(\begin{array}{ccc}
     0\\
     e_{n}^{2}
\end{array}\right) : n \geq 0 \right\rbrace\mathbf ,\,\,\,\, \left\lbrace \left(\begin{array}{ccc}
    e_{-n}^{1}\\
     0
\end{array}\right), \left(\begin{array}{ccc}
     0\\
     e_{-n}^{2}
\end{array}\right) : n \geq 1 \right\rbrace,
\end{equation} 
respectively. 

We have $T = \tilde{T_{1}} \oplus \tilde{T_{2}},$ where $\tilde{T_{i}}$ is an operator on $\tilde{H_{i}}$, $i=1,2$. Also note that $\tilde{H_{i}}$ is invariant under $\pi$. So, $\tilde{T_{1}}$ is a homogeneous operator with associated representation $D_{\lambda + 2a}^{+} \oplus D_{\lambda}^{+}$ and $\tilde{T_{2}}$ is a homogeneous operator with associated representation $D_{2-\lambda}^{-} \oplus D_{2-\lambda+2p}^{-}$.
\end{proof}

\begin{lem}\label{a,b and m are 0}
Suppose $T = \left[\begin{array}{ccc}
    T_{1} & S_{1}\\
    S_{2} & T_{2}
\end{array}\right]$ is a homogeneous operator with associated representation $\pi_{1} \oplus \pi_{2},$ where $\pi_{1} = D_{\lambda}^{+} \oplus D_{2-\lambda}^{-}$ and $\pi_{2} =D_{\lambda}^{+} \oplus D_{2-\lambda+2p}^{-}$ and  $p$ is some positive integer.  Then $T$ is reducible. Furthermore,  $T = \tilde{T_{1}} \oplus \tilde{T_{2}},$ where $\tilde{T_{1}}$ is a homogeneous operator with associated representation $D_{\lambda}^{+} \oplus D_{\lambda}^{+}$ and $\tilde{T_{2}}$ is a homogeneous operator with associated representation $D_{2-\lambda}^{-} \oplus D_{2-\lambda + 2p}^{-}$ or $T = T_1 \oplus T_2$.
\end{lem}

\begin{proof}
Homogeneity of $T$ implies that the operators $T_{i}$ and $S_{i}$ satisfy equations \eqref{eqn:2.1.17}, \eqref{eqn:2.1.18}, \eqref{eqn:2.1.19} and \eqref{eqn:2.1.20}.
Recall  that $T_{1}$ and $T_2$ are weighted shifs with respect to the orthonormal basis $\{e_{n}^{1} \}$ and $\{e_n^2\},$ respectively by virtue of Lemma \ref{lamda is not 1}. Let $\{u_{n}\}$ and $\{v_n\}$ be the corresponding weights for $T_{1}$ and $T_2,$ respectively. Since $p > 0$, it follows form Lemma \ref{lamda is not 1} that $v_{-1} = 0$.

From the equation \eqref{basic eq: 8}, 
we obtain that (i) for all $n \geq -1$ there exist $\beta_{n} \in \mathbb{C}$ such that $S_{2} e_{n}^{1} = \beta_{n} e_{n+1}^{2},$ (ii) for all $n \geq p+2$ there exists $\beta_{-n} \in \mathbb{C}$ such that $S_{2}e_{-n}^{1} = \beta_{-n}e_{-n+p+1}^{2}$ and (iii) $S_{2}e_{-n}^{2} = 0$, for all $1 < n < p+2$.

Applying \hyperlink{algorithm 2}{Algorithm $2$} for $i = n,$ $n \geq -1$ and $j = 0$, we get
\begin{equation*}
\beta_{-1}\left\langle \pi_{1}(\phi_{a}) e_{n}^{1}, e_{-1}^{1} \right\rangle + \beta_{n} \left\langle \pi_{2}(\phi_{a}) e_{n+1}^{2}, e_{0}^{2} \right\rangle = 0.
\end{equation*}
Now, if $n \geq 0$, then from the preceding equation, we find that $\beta_{n} \left\langle \pi_{2}(\phi_{a}) e_{n+1}^{2}, e_{0}^{2} \right\rangle = 0$ and therefore $\beta_{n} = 0$ for all $n \geq 0$.
For $n = -1,$ from the same equation, we have  
\begin{equation*}
\beta_{-1}\left\langle \pi_{1}(\phi_{a}) e_{-1}^{1}, e_{-1}^{1} \right\rangle + \beta_{-1} \left\langle \pi_{2}(\phi_{a}) e_{0}^{2}, e_{0}^{2} \right\rangle = 0.
\end{equation*} 
However, it is easily verified that $\left\langle \pi_{1}(\phi_{a}) e_{-1}^{1}, e_{-1}^{1} \right\rangle + \left\langle \pi_{2}(\phi_{a}) e_{0}^{2}, e_{0}^{2} \right\rangle \not = 0.$ Therefore, $\beta_{-1} = 0$.

Again applying \hyperlink{algorithm 2}{Algorithm $2$} for $i = -1$ and $j = -n+p+1$, $n \geq p+2,$ we observe that
$\beta_{-n} \left\langle \pi_{1}(\phi_{a})e_{-1}^{1}, e_{-n}^{1} \right\rangle = 0.$
Consequently, we have $\beta_{-n} = 0,$ for $n \geq p+2$. This proves that $S_{2}e_{-n}^{1} = 0$, for all $n \geq 2$ and therefore $S_{2} = 0$.

Form the equation \eqref{basic eq: 7}, we have
 (i) for all $n \geq 0,$ there exists $\alpha_{n} \in \mathbb{C}$ such that $S_{1} e_{n}^{2} = \alpha_{n}e_{n+1}^{1}$ and (ii) for all $n \geq 1,$ there exists $\alpha_{-n} \in \mathbb{C}$ such that $S_{1}e_{-n}^{2} = \alpha_{-n}e_{-n-p+1}^{1}$.


Applying \hyperlink{algorithm 1}{Algorithm $1$} for $i = n \geq 0$ and $j = 0,$ we get
$\alpha_{n} \left\langle \pi_{1}(\phi_{a}) e_{n+1}^{1}, e_{0}^{1} \right\rangle = 0.$
Consequently, for all $n \geq 0,$ we see that $\alpha_{n} = 0.$ This proves that $S_{1}e_{n}^{2} = 0,$ $n \geq 0$ .

Again, applying \hyperlink{algorithm 1}{Algorithm $1$} for $i = -n$, $n \geq 1,$ and $j = 0,$ we get
\begin{equation*}
\alpha_{-n}u_{-1} \left\langle \pi_{1}(\phi_{a}) e_{-n-p+1}^{1}, e_{-1}^{1} \right\rangle = 0.
\end{equation*}
It follows  that $\alpha_{-n}u_{-1} = 0,$ $n \geq 1$. Hence if $u_{-1} \neq 0$, then  for all $n \geq 1,$  we see that  $\alpha_{-n} = 0$ and therefore  $S_{1} = 0$. Putting all of these together, we infer that $T = T_{1} \oplus T_{2},$ where $T_{1}$ is a homogeneous operator with associated representation $\pi_{1}$ and $T_{2}$ is a homogeneous operator with associated representation $\pi_{2}$.

Let $\tilde{T_1}, \tilde{T_2}$ be the operators which were constructed in Lemma \ref{b and m are 0}.  If $u_{-1} = 0$, then we have $T = \tilde{T_{1}} \oplus \tilde{T_{2}}.$ The operators  $\tilde{T_1}$ and $\tilde{T_{2}}$ are homogeneous,  and in this case, the associated representations are  $D_{\lambda}^{+} \oplus D_{\lambda}^{+}$  and $D_{2-\lambda}^{-} \oplus D_{2-\lambda+2p}^{-},$ respectively.
\end{proof}

\begin{lem}\label{a,b,m and p are 0}
Let $T = \left[\begin{array}{ccc}
    T_{1} & S_{1}\\
    S_{2} & T_{2}
\end{array}\right]$ be a homogeneous operator with associated representation $\pi_{1} \oplus \pi_{2}$ where $\pi_{1} = D_{\lambda}^{+} \oplus D_{2-\lambda}^{-}$ and $\pi_{2} =D_{\lambda}^{+} \oplus D_{2-\lambda}^{-}$. Then $S_1 = 0$ and $S_2 = 0$.
\end{lem}

\begin{proof}
In this case $\pi_{1} = \pi_{2}$. Denote $\pi_1 = \pi_2 = \pi$ and $e_n^1 = e_n^2 = e_n.$ Homogeneity of $T$ implies that the operators $T_{i}$ and $S_{i}$ satisfy equations \eqref{eqn:2.1.17}, \eqref{eqn:2.1.18}, \eqref{eqn:2.1.19} and \eqref{eqn:2.1.20}.
Repeating an argument similar to the one in Lemma \ref{lamda is not 1}, we find that $T_{1},$ $T_2,$ $S_1$ and $S_2$ are weighted shifts with respect to the orthonormal basis $\{e_{n}\}$. Let $\{u_{n}\}$, $\{v_n\}$, $\{\alpha_n\}$ and $\{\beta_n\}$ be the weights  for $T_{1}$, $T_2$, $S_1$ and $S_2,$ respectively.

Now we prove that $S_{1} = 0$. Applying \hyperlink{algorithm 1}{Algorithm $1$} for $i = n,$ $n \geq 0$ and $j = 0,$ we obtain
$\alpha_{n} \left\langle \pi(\phi_{a})e_{n+1}, e_{0}\right\rangle = 0.$
This implies that $\alpha_{n} = 0,$  $n \geq 0$.

Again applying \hyperlink{algorithm 1}{Algorithm $1$} for $i = -1$ and $j = n,$ $n \leq -1,$ we get
\begin{equation*}
\alpha_{n} \left\langle \pi(\phi_{a}) e_{-1}, e_{n} \right\rangle + \alpha_{-1} \left\langle \pi(\phi_{a}) e_{0}, e_{n+1} \right\rangle = 0.  
\end{equation*}
%
This implies that $\alpha_{n} = 0,$ $n \leq -1,$ proving that $S_1 = 0$.
A similar computation shows that $S_{2} = 0.$  
\end{proof}

\begin{thm}
Suppose $T = \left[\begin{array}{ccc}
    T_{1} & S_{1}\\
    S_{2} & T_{2}
\end{array}\right]$ is a homogeneous operator with associated representation $\pi_{1} \oplus \pi_{2}$ where $\pi_{1} = D_{\lambda + 2a}^{+} \oplus D_{2-\lambda + 2b}^{-}$ and $\pi_{2} = D_{\lambda + 2m}^{+} \oplus D_{2-\lambda + 2p}^{-}$.  Then either $T = \tilde{T_{1}} \oplus \tilde{T_{2}},$ where $\tilde{T_{1}}$ is a homogeneous operator with associated representation $D_{\lambda + 2a}^{+} \oplus D_{\lambda + 2m}^{+}$ and $\tilde{T_{2}}$ is a homogeneous operator with associated representation $D_{2-\lambda + 2b}^{-} \oplus D_{2-\lambda + 2p}^{-}$ or $T = T_1 \oplus T_2$. In particular, $T$ is reducible. 
\end{thm}

\begin{proof}
We divide the proof into several cases and discuss each case separately. Let $\tilde{H_{1}},$  $\tilde{H_{2}}$ be as in \eqref{eqn:2.31a} and 
$\tilde{T_{i}}= T_{|\tilde{H_i}},$  $i=1,2.$ 
\begin{enumerate}
\item[(i)] Assume that none of the $a, b, m, p$ are zero.  Then from Lemma \ref{lamda is not 1}, it follows that $T = \tilde{T_{1}} \oplus \tilde{T_{2}}.$  Also note that $\tilde{H_{i}}$ is invariant under $\pi$. So, $\tilde{T_{1}}$ is a homogeneous operator with associated representation $D_{\lambda + 2a}^{+} \oplus D_{\lambda + 2m}^{+}$ and $\tilde{T_{2}}$ is a homogeneous operator with associated representation $D_{2-\lambda + 2b}^{-} \oplus D_{2-\lambda + 2p}^{-}$.
\item[(ii)] Assume that exactly one of $a, b, m, p$ is non-zero. Then from Lemma \ref{lamda is not 1}, it follows that $T = \tilde{T_{1}} \oplus \tilde{T_{2}}.$ 
\item[(iii)] It follows from Lemma \ref{lamda is not 1} that $T = \tilde{T_{1}} \oplus \tilde{T_{2}}$ if either  $a = 0, b \neq 0, m = 0, p \neq 0$ or $a \neq 0, b = 0, m \neq 0,  p = 0.$
\item[(iv)] The case of  $a \neq 0,  b = 0, p \neq 0, m = 0$ is precisely Lemma \ref{b and m are 0}.  
\item[(v)] Assume that $a = 0, b \neq 0, m \neq 0, p = 0.$ Since  $T^{*}$ is a homogeneous operator with associated representation $\pi_{1}^{\#} \oplus \pi_{2}^{\#},$ the proof follows by applying Lemma \ref{b and m are 0} to  $T^*$. 
\item[(vi)] Assume that $a = 0, b = 0, m \neq 0, p \neq 0$. The associated representation of the operator $T$ is $D_{\lambda}^{+} \oplus D_{2-\lambda}^{-} \oplus D_{\lambda + 2m}^{+} \oplus D_{2-\lambda + 2p}^{-} = \left(D_{\lambda}^{+} \oplus  D_{2-\lambda + 2p}^{-}\right) \oplus \left(D_{\lambda + 2m}^{+} \oplus  D_{2-\lambda}^{-}\right)$. Now, the proof follows form Lemma \ref{b and m are 0}.
\item[(vii)] Assume that $a \neq 0, b \neq 0, m = 0, p = 0$. This is same as (vi).
\item[(viii)] The cases of $a = 0, b = 0, m = 0, p \neq 0$ and $a = 0, m = 0, p = 0, b \neq 0$ are covered in  Lemma \ref{a,b and m are 0}.
\item[(ix)] In case, $b = 0, m = 0, p = 0, a \neq 0$ or  $a = 0, b = 0, p = 0, m \neq 0,$ the proof is completed  by applying the Lemma \ref{a,b and m are 0} to $T^*.$
\item[(x)] Assume $a = 0, b = 0, m =0, p=0.$ This case is exactly Lemma \ref{a,b,m and p are 0}.
\end{enumerate} 
This is an enumeration of all the sixteen possibilities (each of the integers $a,b,m,p$ is either zero or positive) completing the proof. 
\end{proof}

Now we prove that there is no irreducible homogeneous operator with associated representation $\pi := P_{1,0} \oplus D_{1 + 2m}^{+} \oplus D_{1 + 2k}^-$. The representation space of $\pi$ is $H := L^2(\mathbb{T}) \oplus H^{(1+2m)} \oplus H^{(1+2k)}.$

\begin{lem}\label{Descriotion of S_1 and S_2 when pi_1 = P_1,0}
Suppose $T = \begin{bmatrix}
T_1 & S_1\\
S_2 & T_2
\end{bmatrix}$ is a homogeneous operator with associated representation $\pi_1 \oplus \pi_2,$ where $\pi_1 = P_{1, 0}$ and $\pi_2 = D_{1 + 2m}^+ \oplus D_{1 + 2k}^-$, $m, k \geq 0$. Then we have the following.
\begin{enumerate}
\item[(a)] The operators $T_{1}$ and $T_2$ are weighted shifts with respect 
to the orthonormal basis $\{e_{n}^{1} \}$ and $\{e_n^2\}$ with weights 
$\{u_{n}\}$ and $\{v_n\},$ respectively. Also $T_2 e_{-1}^2 = 0$ except when $m=0$ and $k=0$.
\item[(b)] If $k \geq 1$, then for all $n \geq 0,$ $S_1 e_n^2 = 0$,  and for all $n \geq 1$, $S_1e_{-n}^2 = \alpha_{-n}e_{-n-k+1}^1$ such that $u_{-1} \alpha_{-n} = 0$ where $\alpha_{-n} \in \mathbb{C}$. If $k = 0$, then $S_1 e_n^2 = 0$ for all $n \neq -1$ and $S_1 e_{-1}^2 = \alpha_{-1} e_0^1$ for some $\alpha_{-1} \in \mathbb{C}$.
\item[(c)] If $m > 1$, then $S_2 = 0$. If $m = 1$, then  for $n \leq -1,$ $S_2 e_{n}^1 = 0$  and for $n\geq 0,$ there exists $\beta_n \in \mathbb{C}$ such that $S_2 e_{n}^{1} = \beta_{n} e_{n}^2$ and  $u_{-1} \beta_{n} = 0.$  If $m = 0$, then $S_2 e_n^1 = 0$, for all $n \neq -1$ and $S_2 e_{-1}^1 = \beta_{-1} e_0^2$ for some $\beta_{-1} \in \mathbb{C}$.
\end{enumerate}
\end{lem}

\begin{proof}(a) Homogeneity of $T$ implies that the operators $T_{i}$ and $S_{i}$ satisfy equations \eqref{eqn:2.1.17}, \eqref{eqn:2.1.18}, \eqref{eqn:2.1.19} and \eqref{eqn:2.1.20}.
Using the equations \eqref{eqn:2.1.17} and \eqref{eqn:2.1.18}, we find that $T_{1}$ and $T_2$ are weighted shifts with respect to the orthonormal basis $\{e_{n}^{1} \}$ and $\{e_n^2\},$ respectively. Let $\{u_{n}\}$ and $\{v_n\}$ be the weights of $T_{1}$ and $T_2,$ respectively. It is easy to see that that $v_{-1} = 0$ except when $m=0$ and $k=0$.

\noindent (b) From the equation \eqref{basic eq: 5},
it follows that there exists a sequence $\{\alpha_n\}$ such that
\begin{equation}\label{eqn:2.2.6}
S_1 e_n^2 = \alpha_n e_{n+m+1}^1, n \geq 0\,\,\mbox{and}\,\,S_1 e_{-n}^2 = \alpha_{-n} e_{-n-k+1}^1, n \geq 1.
\end{equation}
Applying \hyperlink{algorithm 1}{Algorithm $1$} for $i = n \geq 0$ and $j = 0,$ we obtain
$$\alpha_n \left\langle \pi_1(\phi_a) e_{n+m+1}^1, e_0^1 \right\rangle = 0.$$
In consequence, $\alpha_n = 0$ for all $n \geq 0.$
\begin{description}
\item[$\boldsymbol{k \geq 1}$] 
Applying \hyperlink{algorithm 1}{Algorithm $1$} for $i = -n,$ $n \geq 1,$ and $j = 0,$ we get
\begin{equation*}
\bar{a}\alpha_{-n}u_{-1} \left\langle \pi_1(\phi_a) e_{-n-k+1}^1, e_{-1}^1 \right\rangle = 0, 
\end{equation*}
which implies that $\alpha_{-n}u_{-1} = 0$ for all $n \geq 1$.

\item[$\boldsymbol{k = 0}$] 
Applying \hyperlink{algorithm 1}{Algorithm $1$} for $i = -1$ and $j = -n+1,$ $n \geq 1,$ we obtain
\begin{equation*}
\alpha_{-n} \left\langle \pi_2(\phi_a) e_{-1}^2, e_{-n}^2 \right\rangle + \alpha_{-1} \left\langle \pi_1(\phi_a) e_{0}^1, e_{-n+1}^1 \right\rangle = 0.
\end{equation*}
This implies that $\alpha_{-n} = 0$ for all $n \geq 2$.
\end{description}

\noindent (c) 
Equation \eqref{basic eq: 6}, in the Appendix, implies that
\begin{enumerate}
\item[(i)] for all $n, n \geq \mbox{max}\{m-1, 0\},$ there exist $\beta_n \in \mathbb{C}$ such that $S_2 e_n^1 = \beta_n e_{-m+n+1}^2$ and for all $n, 0 \leq n < \mbox{max}\{m-1, 0\},$ $S_2 e_n^1 = 0$,
\item[(ii)] for all $n, n \geq k+2,$ there exist $\beta_{-n} \in \mathbb{C}$ such that $S_2 e_{-n}^1 = \beta_{-n} e_{-n+k+1}^2$ and for all $n, \nolinebreak 2 \leq n < k+2,$ $S_2 e_{-n}^1 = 0$,
\item[(iii)] there exists $\beta_{-1} \in \mathbb{C}$ such that $S_2 e_{-1}^1 = \beta_{-1} e_{0}^2$ where $\beta_{-1} = 0$ if $m \neq 0.$
\end{enumerate}

Applying \hyperlink{algorithm 2}{Algorithm $2$} for $i = -1$ and $j = n+k+1,$ $n \geq k+2,$ we see that $\beta_{-n} = 0.$
Thus, we have $S_2e_{-n}^1 = 0$ for all $n \geq 2$.
\begin{description}
\item[$\boldsymbol{m > 1}$] 
Applying \hyperlink{algorithm 2}{Algorithm $2$} for $i = n \geq m-1$ and $j = 0,$ we obtain
\begin{equation*}
\beta_n \left\langle \pi_2(\phi_a) e_{-m+n+1}^2, e_0^2 \right\rangle = 0.
\end{equation*}
Thus, for $n \geq m-1$, $\beta_n = 0.$ Consequently, $S_2 = 0$.

\item[$\boldsymbol{m = 1}$] 
In this case also applying \hyperlink{algorithm 2}{Algorithm $2$}, for $i = -1$ and $j = n \geq 0,$ we obtain
\begin{equation*}
\bar{a}\beta_n u_{-1} \left\langle \pi_2(\phi_a) e_{0}^1, e_n^1 \right\rangle = 0.
\end{equation*}
Thus, for $n \geq 0$, $u_{-1}\beta_n = 0$.

\item[$\boldsymbol{m = 0}$] 
Again applying \hyperlink{algorithm 2}{Algorithm $2$}, for $i = n \geq -1$ and $j = 0,$ we obtain
\begin{equation*}
\beta_{-1} \left\langle \pi_1(\phi_a) e_n^1, e_{-1}^1 \right\rangle + \beta_n \left\langle \pi_2(\phi_a)e_{n+1}^2, e_0^2 \right\rangle = 0.
\end{equation*}
This implies that $\beta_n = 0,$ $n \geq 0.$
\end{description}\vskip -1.5em
\end{proof}

\begin{thm}\label{classification when pi_1 is P_1,0}
Suppose $T = \begin{bmatrix}
T_1 & S_1\\
S_2 & T_2
\end{bmatrix}$ is a homogeneous operator with associated representation $\pi = \pi_1 \oplus \pi_2,$ where $\pi_1 = P_{1, 0}$ and $\pi_2 = D_{1 + 2m}^+ \oplus D_{1 + 2k}^-$, $m, k \geq 0$. Then $T$ is reducible.
\end{thm}

\begin{proof}
By virtue of Lemma \ref{Descriotion of S_1 and S_2 when pi_1 = P_1,0}, it is easy to see that if either $k \neq 0$ or $m \neq 0,$ then $T$ is reducible. Thus to complete the proof, we have to show that $T$ is reducible only when $m = k = 0.$
It follows from Lemma \ref{Descriotion of S_1 and S_2 when pi_1 = P_1,0}(a)  that the operators $T_{1}$ and $T_2$ are weighted shifts with respect to the orthonormal basis $\{e_{n}^{1} \}$ and $\{e_n^2\},$ respectively. Let  $\{u_{n}\}$ and $\{v_n\}$ be the corresponding weights. From Lemma \ref{Descriotion of S_1 and S_2 when pi_1 = P_1,0}(b), we see that for $n \neq -1,$ $S_1 e_{n}^2 = 0$ and $S_2 e_{n}^1 = 0$. Clearly, $\tilde{H_1}$ is invariant under $T$. Let $A := T_{|\tilde{H_1}}$ and $B := PT_{|\tilde{H_2}},$ where  $\tilde{H_2}$  is defined in \eqref{eqn:2.31a} and $P$ is the projection of $H$ onto $\tilde{H_2}$. Since $\tilde{H_1}$ and $\tilde{H_2}$ are invariant under $\pi,$ it follows from \cite[Proposition 2.4]{CONSTCharGMBB} that $A$ and $B$ are homogeneous operators with associated representations $\pi_{|\tilde{H_1}}$ and $\pi_{|\tilde{H_2}},$ respectively. Since $\pi_{|\tilde{H_1}}$ is equivalent to $D_{1}^+ \oplus D_1^+$ and $S_1 e_{n}^2 = 0,$ $S_2 e_{n}^1 = 0$ for all $n \geq 0,$ it follows, using homogeneity of $A$, that $u_n = 1,$ $v_n = 1$ for all $n \geq 0$. Similarly, it follows that $u_n = 1,$ $v_n = 1$ for all $n \leq -2.$
Therefore $T$ must be reducible. 
This completes the proof since we have shown that the operator $T$ is reducible in every possible combination of the associated representation. 
\end{proof}
Since $P_{1,0}$ is not equivalent to the direct sum $ $ as explained in Remark \ref{Remark about redicibility of T when the rep is P_1,s + P_1,0}, the case where the associated representation is $\pi = P_{1,0} \oplus P_{1,0}$ has to be settled separately, which is given in the Theorem below. The proof requires no new idea and is omitted. 
\begin{thm}\label{homog op with associated rep P_10 + P_10}
Suppose $T$ is a homogeneous operator on $L^2(\mathbb{T}) \oplus L^2(\mathbb{T})$ with associated representation $\pi = P_{1,0} \oplus P_{1,0}$. Then $T$ is reducible.
\end{thm}

\section{Conclusion}
We have proved, in Section 5 and Section 6, that if the associated representation of a homogeneous bi-lateral $2$-shift $T$ is a direct sum of either three irreducible  or four irreducible representations,
then the operator $T$ must be reducible. Combining this with the analysis, in Section 4, of the remaining case, where the associated representation is the direct sum of two irreducible Continuous series representations, we obtain the proof of our main theorem stated below. 

\begin{mainT}
$(a)$ The irreducible homogeneous bi-lateral $2$-shifts in  $\mathcal{C}$ (respectively,  in $\mathcal{P}$ and $\mathcal{P}_0$)  are mutually inequivalent.

$(b)$ The three classes of irreducible homogeneous bi-lateral $2$-shifts $\mathcal{C},$  $\mathcal{P}$ and $\mathcal{P}_0$ are mutually inequivalent.

$(c)$ Let $T$ be an irreducible homogeneous bi-lateral $2$-shift. Then, up to unitary equivalence, $T$ is in either $\mathcal{C}$ or $\mathcal{P}$ or $\mathcal{P}_0.$
\end{mainT}

\appendix{\section{Computations and an algorithm}}
Let $\pi_1$ and $\pi_2$ be two projective representation of M\"{o}b on the Hilbert spaces $H_1$ and $H_2,$ respectively, such that one of the following holds:
\begin{itemize}
\item[\bf{(I)}] $\pi_1$ and $\pi_2$ are from the irreducible Continuous series representation.

\item[\bf{(II)}] $\pi_{1}$ is from the Continuous series representations and $\pi_{2} = D_{\lambda + 2m}^{+} \oplus D_{2 - \lambda + 2k}^{-},$ $-1 < \lambda \leq 1;$ $m, k$ are integers.

\item[\bf{(III)}] $\pi_{1} = D_{\lambda + 2a}^{+} \oplus D_{2-\lambda + 2b}^{-}$ and $\pi_{2} = D_{\lambda + 2m}^{+} \oplus D_{2-\lambda + 2p}^{-}$ where $\lambda \in (0, 2]$ and $a, b, m, p$ are any non-negative integers. 
\end{itemize}
Suppose $T_1$ and $T_2$ are bounded operators on $H_1$ and $H_2$, respectively and $S_1 : H_2 \to H_1$ and $S_2 : H_1 \to H_2$ be operators which satisfies the following relations 
\begin{equation}\label{basic eq: 1}
S_{1}\pi_{2}(\phi) - e^{i \theta} \pi_{1}(\phi) S_{1} = \overline{a}T_{1}\pi_{1}(\phi)S_{1} + \overline{a} S_{1} \pi_{2}(\phi) T_{2},\,\,\phi \in \mbox{M\"{o}b}
\end{equation}
and
\begin{equation}\label{basic eq: 2}
S_{2}\pi_{1}(\phi) - e^{i \theta} \pi_{2}(\phi) S_{2} = \overline{a}S_{2}\pi_{1}(\phi)T_{1} + \overline{a} T_{2} \pi_{2}(\phi) S_{2},\,\,\phi \in \mbox{M\"{o}b}.
\end{equation}
\subsection*{(I)} We know that $\{z^{n} : n \in \mathbb{Z}\}$ is an orthogonal basis of $H_{i}$. Let $e_n^i = \frac{z^n}{\|z^n\|_i},\,\,i = 1, 2,$ where $\| \cdot\|_i$ denote the inner product of $H_i.$ The set of vectors $\{e_n^i : n \in \mathbb{Z}\}$ is an orthonormal basis of $H_i.$ Let $\phi_{\theta} \in$ M\"{o}b
be such that $\phi_{\theta}(z) = e^{i \theta}z$. Evaluating equation \eqref{basic eq: 1} on $z^n$ and putting $\phi=\phi_{\theta},$ we obtain 

\begin{equation}\label{basic eq: 3}
\pi_{1}(\phi_{\theta}) S_{1}z^{n} = e^{-i\left(n+1+\frac{\lambda_{2}}{2}\right) \theta} S_{1} z^{n},\,\,n \in \mathbb{Z}.
\end{equation}
Thus the existence of a sequence $\{\alpha_{n} : n \in \mathbb{Z} \}$ such $S_{1} e_{n}^{2} = \alpha_{n} e_{n+1}^{1}$ follows.
Suppose $T_i$'s are weighted shift with respect to the orthonormal basis $\{e_n^i : n \in \mathbb{Z}\}$. Then evaluating equation \eqref{basic eq: 1} on the vector $e_{m}^{2},$ putting $\phi=\phi_{a},$ taking inner product with $e_n^1$ and finally using the matrix coefficient of $\pi_i(\phi_a)$ (see \cite[p. 316]{Homogshift}), we obtain
\begin{multline}\label{basic eq: 4}
\alpha_{n-1} \frac{\|z^{n-1}\|_{2}}{\|z^{m}\|_{2}} |\phi'_{a}(0)|^{\mu_{2}} \displaystyle \sum_{k \geq (m-n+1)^{+}} C_{k}^{2}(m, n-1)r^{k} - \alpha_{m} \frac{\|z^{n}\|_{1}}{\|z^{m+1}\|_{1}} |\phi'_{a}(0)|^{\mu_{1}} \displaystyle \sum_{k \geq (m-n+1)^{+}} C_{k}^{1}(m+1, n)r^{k} \\
= v_{m} \alpha_{n-1} \frac{\|z^{n-1}\|_{2}}{\|z^{m+1}\|_{2}} |\phi'_{a}(0)|^{\mu_{2}} \displaystyle \sum_{k \geq (m-n+2)^{+}} C_{k}^{2}(m+1, n-1)r^{k} \\
 + \alpha_{m} u_{n-1} \frac{\|z^{n-1}\|_{1}}{\|z^{m+1}\|_{1}} |\phi'_{a}(0)|^{\mu_{1}} \displaystyle \sum_{k \geq (m-n+2)^{+}} C_{k}^{1}(m+1, n-1)r^{k},
\end{multline}
where $C_{k}^i(m, n) = \left(\begin{array}{cc}
 -\lambda - \mu_i - m\\
 k + n -m
\end{array}\right)\left(\begin{array}{cc}
 - \mu_i + m\\
 k 
\end{array}\right),$ $i = 1, 2$ and $u_n, v_n$ are weights of $T_1, T_2,$ respectively.
Similar conclusions are true for $S_2$ as well. When $\pi_1 = \pi_2,$ we denote $C_{k}^i$ by $C_k.$

\subsection*{(II)} Let $H^{(\lambda + 2m)}$ be the representation space of $D_{\lambda + 2m}^{+}$ and $H^{(2-\lambda + 2k)}$ be the representation space of $D_{2-\lambda + 2k}^{-}$. Let $H_{2} = H^{(\lambda + 2m)} \oplus H^{(2-\lambda + 2k)}$. Define 
\begin{equation}\label{eq: onb}
e_{n}^{2} := \left(\begin{array}{c}
    \frac{z^{n}}{\|z^{n}\|_{\lambda + 2m}}\\
     0
\end{array}\right),\,\, n \geq 0\,\,\mbox{and}\,\,e_{-n}^{2} := \left(\begin{array}{c}
    0\\
   \frac{z^{n-1}}{\|z^{n-1}\|_{2-\lambda + 2k}}
\end{array}\right),\,\,n \geq 1.
\end{equation}
The set of vectors $\{e_{n}^{2} : n \in \mathbb{Z} \}$ is an orthonormal basis of $H_{2}$. Let $\phi_{\theta}$ be a rotation in M\"{o}b. Then 
$$\pi_{2}(\phi_{\theta}) e_{n}^{2} = e^{-i \left( n + m + \frac{\lambda}{2} \right) \theta}e_{n}^{2},\,n \geq 0\,\,\mbox{and}\,\,\pi_{2}(\phi_{\theta}) e_{-n}^{2} = e^{i \left( n + k - \frac{\lambda}{2} \right) \theta}e_{-n}^{2},\,n \geq 1.$$
Substituting $\phi=\phi_{\theta}$ in the equation \eqref{basic eq: 1} and \eqref{basic eq: 2}, respectively, we obtain  
\begin{equation}\label{basic eq: 5}
\pi_{1}(\phi_{\theta})Se_{n}^{2} = e^{-i \left( n + 1 + m + \frac{\lambda}{2} \right) \theta} Se_{n}^{2},\,\, n \geq 0;\,\, \pi_{1}(\phi_{\theta})Se_{-n}^{2} = e^{i \left( n - 1 + k - \frac{\lambda}{2} \right) \theta} Se_{-n}^{2},\,\, n\geq 1
\end{equation}
and
\begin{equation}\label{basic eq: 6}
\pi_{2}(\phi_{\theta}) S_{2} e_{n}^{1} = e^{-i \left(n + 1 + \frac{\lambda}{2} \right) \theta} S_{2}e_{n}^{1},\,\,n \in \mathbb{Z}. 
\end{equation}

\subsection*{(III)} Substituting $\phi=\phi_{\theta}$ in equation \eqref{basic eq: 1} and \eqref{basic eq: 2}, respectively, we obtain
\begin{equation}\label{basic eq: 7}
\pi_{1}(\phi_{\theta})S_{1}e_{n}^{2} = e^{-i \left(n + 1 + m + \frac{\lambda}{2} \right) \theta} S_{1} e_{n}^{2},n \geq 0;\,\,\,
\pi_{1}(\phi_{\theta})S_{1}e_{-n}^{2} = e^{i \left(n - 1 + p - \frac{\lambda}{2} \right) \theta} S_{1} e_{-n}^{2}, n \geq 1
\end{equation}
and 
\begin{equation}\label{basic eq: 8}
\pi_{2}(\phi_{\theta}) S_{2} e_{n}^{1} = e^{-i \left(n + 1 + a + \frac{\lambda}{2} \right) \theta} S_{2} e_{n}^{1}, n \geq 0;\,\,\,\pi_{2}(\phi_{\theta}) S_{2} e_{-n}^{1} = e^{i \left(n - 1 + b - \frac{\lambda}{2} \right) \theta} S_{2} e_{-n}^{1}, n \geq 1.
\end{equation}
where $e_n^1$ and $e_n^2$ are defined in a similar way as in \eqref{eq: onb}.

The following two algorithms have been used in section $5$ and $6$ repeatedly:
\subsection*{\texorpdfstring{\protect\hypertarget{algorithm 1}{Algorithm 1}}{}}
Substitute $\phi = \phi_a$ in the equation \eqref{basic eq: 1}, evaluate at the vector $e_i^2$ and take inner product with the vector $e_j^1.$
\subsection*{\texorpdfstring{\protect\hypertarget{algorithm 2}{Algorithm 2}}{}}
Substitute $\phi = \phi_a$ in the equation \eqref{basic eq: 2}, evaluate at the vector $e_i^1$ and take inner product with the vector $e_j^2.$

\medskip \textit{Acknowledgments}.
The author would like to express his sincere gratitude to G. Misra for his patient guidance and suggestions in the preparation of this paper. 

\end{document}